\documentclass[10pt,twoside]{article}
\usepackage{float,latexsym,graphicx,shortvrb,fancyhdr}
\usepackage{amsmath,amsfonts,bm,mathtools,caption, marvosym,subcaption}
\usepackage[a4paper]{geometry}
\usepackage{enumerate}
\fancypagestyle{firststyle}
{
   
   \fancyhf{}
   \fancyhead[L]{}
   \fancyhead[R]{}
   \fancyfoot[L]{ \footnotesize Z. Jin \\ Department of Mathematics, Shanghai Maritime University, Shanghai, 201306, P.R.China\\ e-mail: asdjz1234@163.com\\
   \vspace{1em}
   D.Y. Gao(\Letter)\\Faculty of Science and Technology, Federation University Australia, Victoria 3353, Australia\\e-mail: d.gao@federation.edu.au
   }
   }

\makeatletter
\renewcommand{\maketitle}{\bgroup\setlength{\parindent}{0pt}
\begin{flushleft}
  \textbf{\@title}

  \@author
\end{flushleft}\egroup
}
\makeatother

\usepackage{amsmath,amsfonts,bm,mathtools,graphicx,latexsym,shortvrb}
\usepackage{xcolor}
\usepackage{hyperref}
\usepackage{multirow}
\usepackage{colortbl}
\usepackage{url}

\newtheorem{theorem}{Theorem}
\newtheorem{lemma}{Lemma}
 \newtheorem{definition}{Definition}
 \newtheorem{Assumption}{Assumption}
\newcommand{\proof}{\par \noindent \bf Proof: \hspace{0mm} \rm }
\newcommand{\qed}{\hspace*{\fill} $\Box$ \vspace{2ex}}

\def\btheo{\begin{theorem}}
\def\etheo{\end{theorem}}
\def\bprop{\begin{proposition}}
\def\eprop{\end{proposition}}
\def\bexam{\begin{example}}
\def\eexam{\end{example}}
\def\bdefi{\begin{definition}}
\def\edefi{\end{definition}}
\def\blemm{\begin{lemma}}
\def\elemm{\end{lemma}}

\newcommand{\mbb}{\mathbb}
\newcommand{\mcal}{\mathcal}

\newcommand{\sta}{\textrm{sta}}
\def\inv{{-1}}

\def\det{\textrm{det}}
\def\diag{\textrm{diag}}

\def\exp{\textrm{exp}}

\def\rank{\textrm{rank}}

\def\non{\nonumber}
\def\wand{\textrm{ and }}
\def\wotherwise{\textrm{otherwise}}

\def\la{\langle}
\def\ra{\rangle}

\def\[#1\]{\begin{align}#1\end{align}}
\def\bcase{\begin{cases}}
\def\ecase{\end{cases}}
\def\bpmat{\begin{pmatrix}}
\def\epmat{\end{pmatrix}}
\def\bbmat{\begin{bmatrix}}
\def\ebmat{\end{bmatrix}}
\def\beqn{\begin{eqnarray}}
\def\eeqn{\end{eqnarray}}
\def\beqnx{\begin{eqnarray*}}
\def\eeqnx{\end{eqnarray*}}
\def\beq{\begin{equation}}
\def\eeq{\end{equation}}
\def\bitem{\begin{itemize}}
\def\eitem{\end{itemize}}
\def\btheo{\begin{theorem}}
\def\etheo{\end{theorem}}
\def\bblock{\begin{block}}
\def\eblock{\end{block}}
\def\benum{\begin{enumerate}}
\def\eenum{\end{enumerate}}

\def\VV{V}

\def\bb{f}

\def\e{\epsilon}

\def\vP{\varPi}

\def\vX{\varXi}

\def\bbR{\mathbb{R}}

\def\calE{\mathcal{E}}

\def\calP{\mathcal{P}}

\def\calS{\mathcal{S}}

\def\calX{\mathcal{X}}
\def\calY{\mathcal{Y}}

\def\lpa{\left(}
\def\rpa{\right)}

\def\lbc{\left\{}
\def\rbc{\right\}}

\def\half{\frac{1}{2}}

\begin{document}

 \thispagestyle{firststyle}

\title{\bf\Large On modeling and global solutions for   d.c. optimization problems by canonical duality theory}
\maketitle

\vspace{3em}

\noindent{{\bf Zhong Jin $\cdot$ David Y Gao}}

\vspace{3em}

\noindent{\textbf{Abstract}}~
This paper presents a canonical  d.c.  (difference of canonical and convex functions) programming problem, which can be used to model general global optimization problems in complex systems.
It shows that by using the canonical duality theory,  a large class of  nonconvex minimization problems can be equivalently converted to a unified concave maximization problem over a convex domain, which can be solved easily under certain conditions. Additionally, a detailed proof for  triality theory is provided, which can be used to identify local extremal solutions.
Applications are illustrated  and open problems are presented.

\vspace{2em}

\vspace{2em}


\vspace{3em}

\textheight=24cm

\section{Problems and motivation}\label{se:intro}
It is known that in Euclidean space  every continuous global optimization problem on a compact set can be reformulated as a d.c. optimization problem, i.e.
a nonconvex problem which can be described in terms of {\em d.c. functions} (difference of convex functions)  and {\em d.c. sets} (difference of convex sets) \cite{tuy95}. By the fact that any constraint  set can be equivalently relaxed by a  nonsmooth  indicator function,
  general nonconvex optimization  problems can be written in the following  standard d.c. programming form
\begin{equation}\label{dc}
 \min \{ f(x)= g(x)-h(x) \; | \;\; \forall  {x\in {\cal X}}  \},
\end{equation}
where   ${\cal X} = \bbR^n$, $g(x), h(x) $ are convex  proper  lower-semicontinuous functions on $\bbR^n$, and the d.c.  function
$f(x)$  to be optimized is usually called the ``objective function'' in mathematical optimization.
A more general model is that $g(x)$ can be  an arbitrary function  \cite{tuy95}.
Clearly,   this d.c. programming problem is artificial. Although  it  can be used to ``model'' a very wide range of mathematical problems \cite{hiri} and has been studied extensively during the last thirty years (cf. \cite{hor-tho,tao-le,s-s}),
  it comes at a price: it is impossible to have an elegant theory and powerful algorithms for solving this
problem without detailed structures on these arbitrarily given functions.
As the result,   even some very simple d.c. programming problems are considered as NP-hard \cite{tuy95}.
This dilemma is mainly due to the existing gap between mathematical optimization and mathematical physics.

\subsection{Objectivity and multi-scale modeling}
 Generally speaking, the concept of {\em objectivity} used in our daily life means the state or quality of being true even outside of a subject's individual biases, interpretations, feelings, and imaginings (see Wikipedia at \url{https://en.wikipedia.org/wiki/Objectivity\_(philosophy)}). In science, the objectivity is often attributed to the property of scientific measurement, as the accuracy of a measurement can be tested independent from the individual scientist who first reports it, i.e.   an objective function does not depend on observers.
In Lagrange mechanics and continuum physics,   a  real-valued  function $W:\calX \rightarrow \bbR $
  is said to be  objective 
 if and only if   (see \cite{Gao00duality}, Chapter 6)
 \[
 W(x) = W(Rx)  \;\; \forall x \in  \calX, \;\;  \forall R \in  {\cal R},
 \]
where ${\cal R}$ is a special rotation group such that
$R^{-1} = R^T, \;\; \det R = 1,  \;\; \forall R \in {\cal R}$.

Geometrically, an objective function does not depend on the rotation, but only on certain measure
of its variable. The simplest measure  in $\bbR^n$  is the $\ell_2$ norm $\| x \|$, which is an objective function since
$\| R x \|^2 = (R x)^T (Rx) = x^T R^T R x= \|x\|^2$ for all special orthogonal matrix $R \in SO(n)$. By Cholesky factorization, any positive definite matrix
has a unique decomposition $C = D^* D$. Thus, any convex quadratic function is objective.
 It was emphasized by P.G. Ciarlet in his recent nonlinear analysis  book \cite{ciarlet}  that the objectivity
 is not an assumption, but an axiom. Indeed, the objectivity is also known as
the {\em axiom of  frame-invariance} in continuum physics (see page 8 \cite{marsd-hugh} and page 42 \cite{tru-noll}).
 Although the objectivity has been well-defined in mathematical physics,
it  is still subjected to seriously study   due to its importance in mathematical modeling (see \cite{liu,murd,murd05}).

Based on the original concept of objectivity, a  multi-scale mathematical model  for general nonconvex systems
was proposed by Gao  in \cite{Gao00duality,gao-opl16}:
\begin{equation}\label{p:gao}
(\mathcal{P}):~~~~\inf \{ \Pi(x)= W(Dx) - F(x)  \; | \;\; \forall  {x \in {\cal X} }   \},
\end{equation}
where  $\calX $ is a feasible space;
$F: {\cal X} \rightarrow \bbR \cup \{ -\infty\}$ is  a so-called {\em subjective function},
   which is  linear  on its effective domain $\calX_a \subset \calX$, wherein,
  certain ``geometrical  constraints'' (such as boundary/initial conditions, etc) are given;
correspondingly,   $W: {\cal Y}  \rightarrow \bbR\cup \{ \infty\} $ is  an
{\em objective function} on its effective domain $\calY_a \subset \calY$,
in which, certain physical constraints (such as constitutive laws, etc) are given;
$D: {\cal X}  \rightarrow {\cal Y} $ is a linear operator which assign each decision variable
in configuration space ${\cal X}$ to an internal variable $y \in {\cal Y} $ at different scale.
  By Riesz representation theorem,  the subjective function
 can be written as
  $F(x) = \la x, \bar{x}^* \ra \;\; \forall x \in \calX_a$,
  where $\bar{x}^* \in \calX^*$ is a given input (or source),
the bilinear form $\la x, x^* \ra :\calX \times \calX^* \rightarrow \bbR$ puts $\calX $ and $\calX^*$ in duality.
Additionally, the positivity conditions $W(y) \ge 0 \;\;\forall y \in \calY_a $, $F(x) \ge 0 \;\;\forall x \in \calX_a$
and coercivity condition $\lim_{\|y\| \rightarrow \infty} W(y) = \infty$
are needed for the target function $\Pi(x)$ to be  bounded below on
its effective domain $\calX_c = \{ x \in \calX_a | \;\; Dx \in \calY_a\}$ \cite{gao-opl16}.
Therefore, the extremality condition $  0 \in \partial \Pi(x)  $ leads to the equilibrium equation  \cite{Gao00duality}
\[
0 \in   D^* \partial W(Dx) - \partial F (x)  \;\; \Leftrightarrow \;\;   D^* y^* - x^* = 0  \;\; \forall  x^* \in  \partial F(x), \;\; y^* \in  \partial W(y).
\]

In this model, the  objective  duality relation  $y^* \in \partial W(y)$ is governed by the constitutive law, which depends only
on mathematical modeling of the system;
  the subjective duality relation $x^* \in \partial F(x)$ leads to the  input   $\bar{x}^* $ of the system,
 which depends only on each given problem.
Thus,   $(\mathcal{P})$ can be used to model general problems in multi-scale complex systems.

\subsection{Real-world problems}
In management science the  variable  $x \in \calX_a \subset \bbR^n$ could  represent  the products of a manufacture company.
  Its dual variable $\bar{x}^* \in \bbR^n$ can be considered as market price (or demands). Therefore, the subjective function
  $ F(x) = x^T \bar{x}^* $ in this example is  the total income of the company.
The products are produced by workers $y  \in \bbR^m$. Due to the cooperation, we have $y =  D x$ and
$ D \in \bbR^{m\times n}$ is a matrix. Workers are paid by salary $y^*= \partial  W(y)$,
therefore, the objective function  $ W(y)$ in this example is the cost.
Thus, $ \Pi(x) = W( D x) -  F(x)$ is the {\em total loss  or target}  and the minimization problem
 $({\cal P})$ leads to
the equilibrium equation
$ D^T \partial_{y} W( Dx) = \bar{x}^*$.
The cost function  $ W( y)$ could be  convex for a very small company, but
usually nonconvex for big companies.

In Lagrange mechanics, the variable $x \in \calX = {\cal C}^1[I; \bbR^n]$ is a continuous  vector-valued function of time $t \in I \subset \bbR$, its components $\{x_i (t) \} (i=1, \dots, n) $
are  known as the  Lagrange coordinates.
 The subjective function in this case
 is a linear functional $ F(x) =   \int_I  x(t)^T \bar{x}^*(t)  dt$, where $ \bar{x}^*(t) $ is a given external force field.
  While  $ W( D x)$ is the so-called   action:
  \[
  W( D x) = \int_I L(x, \dot{ x} ) dt , \;\; L=  T(\dot{x}  ) - V(x),
  \]
  where $ T $ is the kinetic energy density,  $ V $ is the potential density, and $L=  T -  V$ is the standard
  {\em Lagrangian density} \cite{l-l}.
  The linear operator $ D x  = \{ \partial_t, 1 \} x = \{ \dot{x}, \; x\}$ is a vector-valued mapping.
  The kinetic energy $T $  must be an objective function of the velocity (quadratic for Newton's mechanics and convex for  Einstein's relativistic theory) \cite{Gao00duality},
  while the potential density $ V$ could be either convex or nonconvex,
  depending  on each problem.
   Together, $ \Pi(x) =  W( D x) -  F(x)$   is called {\em total action}.
   The extremality condition $\partial \Pi(x) = 0$ leads to the well-known Euler-Lagrange equation
   \[
 D^* \partial W(D x) = \partial^*_t  \frac{d  T(\dot{x}) }{d \dot{x} }- \frac{d V(x) }{d x} = \bar{x}^* ,
   \]
   where $\partial^*_t $ is an adjoint operator of $\partial_t$.
For convex Hamiltonian systems, both  $T$ and $V$
  are convex, thus, the least  action principle leads to a typical  d.c. minimization  problem
\[
\inf  \{\Pi(x) = K(\partial_t x ) - P(x) \} , \;\;   K(y) = \int_I T(y) dt , \;\; P(x) = \int_I  [ V(x)  + x^T \bar{x}^* ] dt ,
\]
where $K(y)$ is the kinetic  energy and $P(x)$ is the total potential energy.
The duality theory for this d.c. minimization problem was first studied by J. Toland  \cite{toland} with successful application in nonlinear
heavy rotating chain, where
\[
K(y) = \int_0^1 \frac{1}{2 \lambda} y^2 dt , \;\; P(x) = \int_0^1 [ x(t) ^2  + t^2 ]^{{1}/{2}} d t,
\]
and the parameter  $\lambda > 0$ depends on angular speed. Clearly, in this application, $K(v)$ is quadratic while $P(x)$ is approaching to linear when $\|x (t) \|$ is sufficiently large. Therefore, the total action $\Pi(x)$ is bounded below and coercive
on $\calX = \{ x(t) \in W^{1,2}[0,1] | \;\; x(0) = 0\}$, the problem $({\cal P})$
has  a unique stable global minimizer.

However,  if  both $K(v)$ and $P(x)$ are quadratic functions (for example,  the classical linear mass-springer system),
the  d.c. minimal problem  $({\cal P})$  will have no stable global minimizer.
It was proved in \cite{Gao00duality}  that, in addition to the double-min duality
\[
\inf_{x \in \calX} \{  \Pi(x) = K(\partial_t x) - P(x) \} = \inf_{y^* \in {\cal Y}^*}   \{ \Pi^*(y^*)
= P^*(\partial^*_t y^*) - K^*(y^*) \} ,
\]
the double-max duality
\[
\sup_{x \in \calX} \{  \Pi(x) = K(\partial_t x) - P(x) \} = \sup_{y^* \in {\cal Y}^*}  \{ \Pi^*(y^*) = P^*(\partial^*_t y^*) - K^*(y^*) \}
\]
holds alternatively, i.e.  the system is in stable periodic vibration on its time domain $I$.
Therefore, this double-min duality reveals an important truth in convex Hamiltonian systems:
 the least action principle is a misnomer for periodic vibration (see Chapter 2 \cite{Gao00duality}).

Now let us consider  another example in  buckling analysis of Euler beam:
\[
\inf  \{ \Pi(u) = K(\partial_{xx} u ) - P(\partial_x u) \},
\]
where both   the bending energy $K$ and the axial strain energy $P$ are quadratic
\[
K(\partial_{xx} u ) = \int_I \frac{1}{2} \alpha u_{xx}^2 d x , \;\; P(\partial_x u) =  \int_I \frac{1}{2}  \lambda u_{x}^2 d x
\]
and $\alpha > 0$ is a constant, $\lambda > 0$ is a given axial load at the end of  the  beam.
Clearly, if $\lambda < \lambda_c$, the eigenvalue of the Euler beam defined by
\[
\lambda_c = \inf \frac{\int_I  \alpha  u_{xx}^2 d x }{\int_I u_x^2 dx},
\]
the d.c.  functional $\Pi(u)$ is convex and the problem $({\cal P})$ has a unique solution.
In this case, the Euler beam is in pre-buckling state.
However,   $\Pi(u)$ is concave if the axial load $\lambda > \lambda_c$ and in this case, we have $\inf \Pi(u) = - \infty$, which means that  the Euler beam is collapsed.
This example shows that  the linear  Euler beam can be used only for pre-buckling problems.
Generally speaking, unconstrained quadratic d.c. programming problem  does not make any physical sense unless it is convex.

In order to study the post-bifurcation problems,
 a nonlinear beam model was proposed by Gao \cite{gao-mrc96}. Instead of quadratic function, the stored energy
 $K$ in this nonlinear model  is a fourth-order polynomial
\beq
K(\partial_{xx} u ) = \int_I\left[ \frac{1}{2} \alpha  u_{xx}^2  + \frac{1}{12} \beta u_{x}^4 \right] d x ,
\eeq
 where $\beta > 0$ is a material constant.  Clearly, if  $\lambda_p = \lambda - \lambda_c < 0$,
  $\Pi(u) = K(\partial_{xx} u) - P(\partial_x u) $
 is strictly convex.  In this case, the problem $(\calP)$ has a unique minimizer and
 the beam is in  pre-buckling state.
 If $\lambda_p > 0$,  the total potential $\Pi(u)$ is nonconvex which has two equally valued local minimizers and one local maximizer.
 Therefore, this nonlinear beam  can be used  to model
post-buckling phenomena, which has been subjected to seriously  study in recent years
(cf.  \cite{ahn,cai-gao-qin,kuttler-etal}).
If the beam is subjected a lateral distributed load $q(x)$,
then we have the following d.c  variational problem
\beq
\inf \{ \Pi(u) = W(Du) - F(u)\} ,
\eeq
where
\[
   W(Du)  = \int_I \left[   \frac{1}{12} \beta  u_{x}^4  - \frac{1}{2}  \lambda_p u_{x}^2  \right]d x ,\;\;
F(u) = \int_I   q  u \;   d x .
\]
The objective function $W(Du) $ in this problem is the so-called {\em double-well  potential}, which
 appears extensively in real-world problems, such as phase transitions, shape-memory alloys, chaotic dynamics and theoretical physics \cite{gao-amma03}.
The subjective function $F(u)$  breaks the symmetry of this nonlinear buckling beam model
 and leads to one global minimizer,  corresponding to a stable buckled state,
  one local minimizer, corresponding to one unstable buckled state,
and one local maximizer, corresponding to one unbuckled state  \cite{cai-gao-qin,santos-gao}.
By  finite element method  the domain $I$ is discretized into $m$-elements $\{ I_ e\}$ such that
 the unknown function can be
 piecewisely  approximated as  $ u (x) \simeq N_e(x) p_e  $ in each element $ I_e $ with $p_e$ as  nodal variables.
 Then, the nonconvex variational problem (15)
 can be numerically reformulated to  the
  d.c. programming problem (1)  with  $g(p)$  as a fourth-order polynomial and  $h(p)$ a quadratic function so that
   $\Pi(p) = g(p) - h(p)$ is bounded below to  have a global minimum solution \cite{cai-gao-qin,santos-gao}.

All the  real-world applications discussed in this section  show a simple fact, i.e.
 the  functions $g(x)$  and $h(x)$ in the
standard d.c. programming problem (\ref{dc}) can't be arbitrarily given, they
must obey certain fundamental laws in  physics   in order to  model real-world systems.
By the facts that  the subjective function $F(x) = \la x , \bar{x}^* \ra $
is necessary for any given real-world   system  in order to have non-trivial solutions (states or outputs) and the
 function $g(x)$ in the standard d.c. programming (\ref{dc}) can  be generalized to a nonconvex function (see Equation (36)
 in \cite{tuy95}), it is reasonable
 to assume that $g(x)$ in  (\ref{dc}) is a general nonconvex function $W(Dx)$ and
   $h(x)$  is a quadratic function
 \[
 Q(x) = \frac{1}{2}  \la  x ,  C x  \ra +  \la x,  \bb \ra ,
 \]
 where  $C: \calX \rightarrow \calX^*$ is a given symmetrical positive definite operator (or matrix)
  and   $\bb \in \calX^*$
 is a given input. Then, the  standard d.c. programming (\ref{dc}) can be generalized to the following form
 \[
 (\calP_{dc}): \;\; \min \{W(Dx) - Q(x) \;| \;\; \forall x \in\calX \}.
 \]

\subsection{Canonical duality theory and goal}
Canonical duality-triality is a  breakthrough  theory  which can be used not only for modeling complex systems within a unified framework, but also for solving real-world  problems with a unified methodology \cite{gao-opl16}.
This theory was developed  originally from Gao and Strang's work in nonconvex mechanics
 \cite{gao-strang1989} and has been applied successfully for solving a large class of challenging problems
 in both nonconvex analysis/mechancis and global optimization, such as phase transitions in solids \cite{gao-yu}, post-buckling of large deformed beam \cite{santos-gao},
 nonconvex polynomial minimization problems with box and integer constraints \cite{gao2005,gao2007,gao-ruan2009}, Boolean and multiple integer programming \cite{fang-gao2007,wang-fang-gao2008}, fractional programming \cite{fang-gao2009}, mixed integer programming\cite{gao-ruan2010}, polynomial optimization\cite{gao2006}, high-order polynomial with log-sum-exp problem\cite{gao-chen2014}.
A comprehensive review on this theory and breakthrough
 from recent challenges are given in \cite{bridgeMMS}.

The goal of this paper is to apply the canonical duality  theory for solving the challenging d.c. programming problem
  (\ref{dc}).
The rest of this paper is arranged as follows.
Based on the concept of objectivity, a canonical d.c. optimization problem and its  canonical dual
  are formulated in the next section.
  Analytical solutions and triality theory for  a general d.c. minimization problem with
   sum of nonconvex polynomial and exponential functions  are discussed in Sections 3 and 4.
  Five special examples  are illustrated  in Section 5. Some conclusions and future work are given in Section 6.

\section{Canonical d.c. problem and its canonical dual}
It is known that the linear operator $D:\calX \rightarrow \calY$ can't change the nonconvex  $W(Dx)$ to a convex function.
According to the definition of the objectivity, a nonconvex function $W:\calY \rightarrow \bbR $ is objective if and only if
there exists a function $\VV:\calY \times \calY \rightarrow \bbR$ such that $W(y) = \VV(y^T y)$ \cite{ciarlet,gao-opl16}.
Based on this fact,  a reasonable assumption can be made for the general problem $(\calP_{dc})$.
\begin{Assumption}[Canonical   Transformation and Canonical Measure]  $\;$ \newline
For a given nonconvex function $W:\calY \rightarrow \bbR \cup \{\infty\}$,  there exists a nonlinear    mapping
$\Lambda:\calX \rightarrow \calE$ and a  convex, l.s.c function  $\VV:\calE  \rightarrow \bbR \cup \{ \infty\} $ such that
\beq
W(Dx) = \VV( \Lambda(x)).   \label{eq-ct}
\eeq
\end{Assumption}

The nonlinear transformation (\ref{eq-ct}) is actually  the {\em canonical transformation}, first
introduced  by Gao in 2000 \cite{gao-jogo00},
 and  $\xi = \Lambda(x)  $ is called a  {\em canonical  measure}.
The canonical measure $\xi = \Lambda(x)$ is also called the {\em
geometrically admissible measure} in
the canonical duality theory \cite{gao-jogo00}, which is not necessarily to be objective.
But the most simple canonical measure in $\bbR^n$  is the quadratic function $\xi = x^T x$,
which is clearly objective. Therefore, the canonical function can be viewed as a generalized objective function.
 Thus,  based on Assumption 1, the generalized d.c. programming problem $(\calP_{dc})$ can be
  written in a canonical d.c.  minimization problem form  ($(\calP)$ for short):
\beq
(\calP): \;\; \min \left\{ \Pi(x) = \VV(\Lambda (x) )  -   Q (x)  | \;\; x \in \calX  \right\}.
\eeq

 Since the canonical measure  $\xi = \Lambda(x) \in \calE$ is  nonlinear
  and $\VV(\xi)$ is convex on $\calE$, the composition   $\VV(\Lambda(x))$ has a higher order nonlinearity   than $Q(x)$. Therefore, the coercivity for the target function $\Pi(x)$ should be  naturally satisfied, i.e.
 \beq
 \lim_{\| x \| \rightarrow \infty } \{\Pi (x) = \VV(\Lambda(x) ) - Q(x) \} = \infty , \label{eq-coe}
 \eeq
 which is a sufficient condition for existence of a global minimal solution to $(CDC)$ (otherwise, the set $\calX$ should be bounded).
 Clearly, this generalized d.c. minimization problem can be used to model a  reasonably large class of    real-world problems
in mathematical physics  \cite{Gao00duality,gao-amma03}, global optimization \cite{gao-cace09},
and computational sciences \cite{bridgeMMS}.

 By the fact that $\VV(\xi)$ is convex, l.s.c. on $\calE$, its conjugate can be uniquely defined by the Fenchel transformation
 \[
 \VV^*(\xi^*) = \sup \{ \la \xi ; \xi^* \ra - \VV(\xi) | \;\; \xi \in \calE \}.
 \]
The bilinear form $\la \xi ; \xi^* \ra$ puts $\calE$ and $\calE^*$ in duality.
According to convex analysis (cf. \cite{eke-tem}), $\VV^*:\calE^* \rightarrow \bbR \cup \{ + \infty\}$ is also convex, l.s.c.  on its domain $\calE^*$ and the following generalized canonical duality relations
\cite{gao-jogo00}  hold on $\calE \times \calE^*$
\[
\xi^* \in \partial \VV(\xi) \;\; \Leftrightarrow \;\; \xi \in \partial \VV^*(\xi^*) \;\; \Leftrightarrow \;\;
\VV(\xi) + \VV^*(\xi^*) = \la \xi ; \xi^* \ra .
\]
 Replacing $ \VV(\Lambda(x)) $ in  the target function $\Pi(x)$  by the Fenchel-Young equality
 $\VV(\xi)  =  \la \xi ; \xi^* \ra - \VV^*(\xi^*)$,
   Gao and Strang's  total complementary function  (see \cite{gao-jogo00})
   $\Xi: \calX \rightarrow \calE^* \rightarrow \bbR \cup \{ - \infty\}$  for this
   (CDC) can be obtained as
 \beq
 \Xi(x, \xi^* ) = \la \Lambda(x)  ; \xi^* \ra - \VV^*(\xi^*) - Q(x)  .
 \eeq
 By this total complementary function, the canonical dual of $\Pi(x)$ can be obtained as
 \beq
 \Pi^d(\xi^*) =  \inf \{ \Xi(x, \xi^*) | \;\; x \in \calX \} = Q^{\Lambda}(\xi^*) - \VV^*(\xi^*),
 \eeq
 where $Q^\Lambda:\calE^* \rightarrow \bbR \cup\{ - \infty\}$ is the so-called $\Lambda$-conjugate of $Q(x)$ defined by
 (see \cite{gao-jogo00})
 \[
 Q^\Lambda(\xi^*) =  \inf \{ \la \Lambda(x) ;  \xi^* \ra - Q(x) \; | \;\; x \in \calX \}.
 \]
  If this $\Lambda$-conjugate has a non-empty  effective domain, the following
   canonical duality
   \beq
   \inf_{x \in \calX} \Pi(x) = \sup_{\xi^* \in \calE^*} \Pi^d(\xi^*)
   \eeq
   holds under certain conditions,  which will be illustrated in the next section.

    \section{ Application and analytical solution}
Let us consider a special application in $\bbR^n$ such  that
\beq
g(x)=  W(Dx) = \sum_{i=1}^p\exp\lpa\frac{1}{2}x^T A_i
x-\alpha_i\rpa +  \sum_{j=1}^r\frac{1}{2}\lpa\frac{1}{2}x^T B_jx-\beta_j\rpa^2,
\eeq
where $ \{ A_i \}_{i=1}^p \in\mathbb{R}^{n\times n}$ are symmetric matrices
and $\{B_j \}_{j=1}^r \in\mathbb{R}^{n\times n}$ are  symmetric positive definite matrices,
 $\alpha_i$ and $\beta_j$ are real numbers.
   Clearly, $g:\bbR^n \rightarrow \bbR$ is nonconvex and highly nonlinear. This type of nonconvex function covers many real applications.

  The canonical measure in this application can be given  as
$$
\xi=\bpmat \theta\\ \eta \epmat=\Lambda(x)=\bpmat
\lbc\frac{1}{2}x^T A_ix\rbc_{i=1}^p\vspace{.2cm}\\
\lbc\frac{1}{2}x^T B_jx\rbc_{j=1}^r
\epmat
~:~\mathbb{R}^n\rightarrow\mcal{E}_a\subseteq\mathbb{R}^m
$$
where $m=p+r$. Therefore, a canonical function can be  defined on  $\calE_a$:
$$
V(\xi)=V_1(\theta)+V_2(\eta)
$$
where
\begin{eqnarray}
&&V_1(\theta)= \sum_{i=1}^p\exp\lpa\theta_i-\alpha_i\rpa,\non\\
&&V_2(\eta)=\sum_{j=1}^r\frac{1}{2}(\eta_j-\beta_j)^2.\non
\end{eqnarray}
Here $\theta_i$ and $\eta_j$ denote the $i$th component of $\theta$ and the $j$th component of $\eta$, respectively. Since $V_1(\theta)$ and $V_2(\eta)$ are convex,
$V(\xi)$ is a convex function. By Legendre transformation, we have the following equation
\beq\label{eq:legendre}
V(\xi)+V^*(\zeta)=\xi^T\zeta,
\eeq
where
$$
\zeta=\bpmat \tau\\\sigma \epmat
=\bpmat  \nabla V_1(\theta)\\\nabla V_2(\eta) \epmat
=\bpmat \lbc
\exp\lpa\theta_i-\alpha_i\rpa
\rbc_{i=1}^p\vspace{0.1cm}\\
 \lbc\eta_j-\beta_j\rbc_{j=1}^r\epmat ~:~\mcal{E}_a\rightarrow \mcal{E}_a^* \subset \bbR^m
$$
and $V^*(\zeta)$ is the conjugate function of $V(\xi)$, defined as
\[
V^*(\zeta)=V_1^*(\tau)+V_2^*(\sigma)
\]
with
\begin{eqnarray}
&&V_1^*(\tau)=\sum_{i=1}^p\lpa\alpha_i+\ln(\tau_i)-1\rpa\tau_i,\label{eq:funV1star}\non\\
&&V_2^*(\sigma)=\half\sigma^T\sigma+\beta^T\sigma,\label{eq:funV2star}\non
\end{eqnarray}
where $\beta=\{\beta_j\}$.

 Since the canonical measure in this application is a quadratic operator, the total complementary function
$\vX : \mbb{R}^n\times \mcal{E}_a^*\rightarrow\mbb{R}$ has the following form
\[\label{eq:totalcompf}
\vX(x,\zeta)
=\frac{1}{2}x^T G(\zeta)x-f^Tx-V_1^*(\tau)-V_2^*(\sigma),
\]
where
$$
G(\zeta)= \sum_{i=1}^p\tau_i A_i+\sum_{j=1}^r\sigma_j B_j-C.
$$
Notice that for any given $\zeta$, the total complementary function $\vX(x,\zeta)$ is a quadratic function of $x$ and its stationary points are the solutions of the following equation
\begin{equation}\label{eq:partialXi}
\nabla_{x}\vX(x,\zeta)= G(\zeta)x-f=0.
\end{equation}
If $\det(G(\zeta))\neq 0$ for a given $\zeta$, then  (\ref{eq:partialXi})  can be solved analytically to have a unique solution
$x= G(\zeta)^{\inv}f$.
Let
\[
\mcal{S}_a=\lbc \zeta\in\mcal{E}_a^*| \; ~  \det(G(\zeta))\neq0 \rbc .
\]
Thus,  on $\mcal{S}_a$ the canonical dual function $\vP^d(\zeta)$ can then be written explicitly as
\[\label{eq:dualfun}
\vP^d(\zeta)=&-\frac{1}{2}f^T G(\zeta)^{\inv}f -  V_1^*(\tau)-V_2^*(\sigma).
\]
 Clearly, both $\vP^d(\zeta)$ and its domain  $\mcal{S}_a$ are nonconvex.
The canonical dual problem is to find all stationary points of $\vP^d(\zeta)$ on its domain, i.e.
\begin{equation}\label{p:dual}
(\mcal{P}^d):~~~~\sta\lbc\vP^d(\zeta)~|~\zeta\in\mcal{S}_a\rbc.
\end{equation}
\begin{theorem}[Analytic Solution and Complementary-Dual Principle]\label{th:AnalSolu} $\;$\newline
Problem ($\calP^d$) is canonical dual to the problem ($\calP$) in the sense that if $\bar{\zeta}\in\mcal{S}_a$ is a stationary  point of $\vP^d(\zeta)$, then
\beq\label{eq:solvedx}
\bar{x}=G(\bar\zeta)^{\inv}f
\eeq
is a stationary  point of $\vP(x)$, the pair $(\bar x,\bar\zeta)$ is a stationary point of $\vX(x,\zeta)$, and we have
\beq\label{eq:nogap}
\vP(\bar{x})=\vX(\bar x,\bar\zeta)=\vP^d(\bar{\zeta}).
\eeq
\end{theorem}

The proof of this theorem  is  analogous with that in \cite{Gao00duality}.
Theorem \ref{th:AnalSolu} shows that  there is no duality gap between the primal problem ($\calP$) and the canonical dual problem ($\calP^d$).

\section{Triality theory}\label{se:triality}
In this section we will study  global optimality conditions for the critical solutions of  the primal and dual problems. In order to identify both global and local extrema of both two problems, we let
\begin{eqnarray*}
&&\mcal{S}_a^+= \lbc\zeta\in\mcal{S}_a~|~G (\zeta) \succ 0\rbc,\label{eq:Saplus}\\
&&\mcal{S}_a^-= \lbc\zeta\in\mcal{S}_a~|~G (\zeta) \prec 0\rbc.\label{eq:Saminus}
\end{eqnarray*}
where $G\succ 0$ means that $G$ is a positive definite matrix and where $G\prec 0$ means that $G$ is a negative definite matrix. It is easy to  prove that both $\calS_a^+$ and $\calS_a^-$ are convex sets and
\[
Q^\Lambda(\zeta) = \inf \{ \la \Lambda(x) ; \zeta \ra - Q(x) | \;\; x \in \bbR^n \} = \left\{
\begin{array}{ll}
 -\frac{1}{2}f^T G(\zeta)^{\inv}f \;\; & \mbox{ if } \zeta \in \mcal{S}_a^+\\
 -\infty &  \mbox{ otherwise }
 \end{array} \right.
 \]
This shows that  $\mcal{S}_a^+$ is an effective domain of $Q^\Lambda(\zeta)$.

For convenience, we first give the first and second derivatives of functions $\vP(x)$ and $\vP^d(\zeta)$:
\begin{eqnarray}
&& \nabla \vP(x)= G (\zeta) x - f,\label{eq:der1Pi}\\
&& \nabla^2\vP(x)= G+Z_0 H Z_0^T,\label{eq:der2Pi}\\
&& \nabla \vP^d(\zeta)=
\left(\begin{array}{l}
\lbc\frac{1}{2}f^T G^\inv A_i G^{\inv}f-\alpha_i-\ln(\tau_i)\rbc_{i=1}^p\vspace{.2cm}\\
\lbc\frac{1}{2}f^T G^\inv B_j G^{\inv}f-\sigma_j-\beta_j\rbc_{j=1}^r
\end{array}\right),\label{eq:der1Pid}\\
&& \nabla^2\vP^d(\zeta)=- Z^T G^\inv Z-H^\inv,\label{eq:der2Pid}
\end{eqnarray}
where $Z_0,Z\in\mbb{R}^{n\times m}$ and $H\in\mbb{R}^{m\times m}$ are defined as
\begin{eqnarray*}
&& Z_0=
\begin{bmatrix}
 A_1x,\ldots, A_px, B_1x,\ldots, B_rx
\end{bmatrix},\label{eq:matrixF}\\
&&Z=
\begin{bmatrix}
 A_1G^{\inv}f,\ldots, A_pG^{\inv}f, B_1G^{\inv}f,\ldots, B_rG^{\inv}f
\end{bmatrix},\label{eq:matrixF}\\
&& H=
\begin{bmatrix}
\diag(\tau)&0\\
0 &E_n
\end{bmatrix} , \label{eq:matrixH}
\end{eqnarray*}
where $E_n$ is a $n\times n$ identity matrix. By the fact that
  $ \tau > 0$,
the matrix $H^\inv$ is positive definite.

Next we can get the lemma as follows whose proof is trivial.
\begin{lemma}
If $M_1,M_2,\ldots,M_N\in\mbb{R}^{n\times n}$ are symmetric positive semi-definite matrices, then $M=M_1+M_2+\ldots+M_N$ is also a positive semi-definite matrix.
\end{lemma}

\begin{lemma}
If $\lambda_{G}$ is an arbitrary eigenvalue of $G$, it follows that
$$\lambda_{G}\geq \sum_{i=1}^p\tau_i \lambda^{A_i}_{min}+\sum_{j=1}^r\sigma_j\bar\lambda^{B_j}-\lambda^{C}_{max},$$
in which $\lambda^{A_i}_{min}$ is the smallest eigenvalue of $A_i$, $\lambda^{C_i}_{max}$ is the largest eigenvalue of $C_i$, and
\begin{equation}\label{eq:111}
\bar\lambda^{B_j}=
    \left\{\begin{array}{ll}
    \lambda^{B_j}_{min},&~~\sigma_j> 0\vspace{.3cm}\\
    \lambda^{B_j}_{max},&~~\sigma_j\leq 0,
    \end{array}\right.
    \end{equation}
where $\lambda^{B_j}_{min}$ and $\lambda^{B_j}_{max}$ are the smallest eigenvalue and the largest eigenvalue of $B_j$ respectively.
\end{lemma}
\begin{proof} Firstly, we need prove $\tau_i(A_i-\lambda^{A_i}_{min}E_n)$, $\lambda^{C}_{max}E_n-C$ and $\sigma_j(B_j-\bar\lambda^{B_j}E_n)$ are all symmetric positive semi-definite matrices.
\begin{enumerate}[(a)]
\item As $\lambda^{A_i}_{min}$ is the smallest eigenvalue of $A_i$, then $A_i-\lambda^{A_i}_{min}E_n$ is symmetric positive semi-definite, so $\tau_i(A_i-\lambda^{A_i}_{min}E_n)$ is symmetric positive semi-definite with $\tau_i=\exp\lpa\theta_i-\alpha_i\rpa>0$.
\item As $\lambda^{C}_{max}$ is the largest eigenvalue of $C$, then $\lambda^{C}_{max}E_n-C$ is a symmetric positive semi-definite matrix.
\item
\begin{enumerate}[(c.1)]
\item As $\lambda^{B_j}_{min}$ is the smallest eigenvalue of $B_j$, then $B_j-\lambda^{B_j}_{min}E_n$ is symmetric positive semi-definite, so when $\sigma_j> 0$ it holds that $\sigma_j(B_j-\lambda^{B_j}_{min}E_n)$ is symmetric positive semi-definite.
\item As $\lambda^{B_j}_{max}$ is the largest eigenvalue of $B_j$, then $B_j-\lambda^{B_j}_{max}E_n$ is symmetric negative semi-definite, so when $\sigma_j\leq 0$ it holds that $\sigma_j(B_j-\lambda^{B_j}_{max}E_n)$ is symmetric positive semi-definite.
\end{enumerate}
   From (c.1) and (c.2), we know $\sigma_j(B_j-\bar\lambda^{B_j}E_n)$ is always symmetric positive semi-definite.
\end{enumerate}
Then by (a), (b), (c) and Lemma 1, we have $$\sum_{i=1}^p\tau_i(A_i-\lambda^{A_i}_{min}E_n)+\sum_{j=1}^r\sigma_j(B_j-\bar\lambda^{B_j}E_n)+\lambda^{C}_{max}E_n-C$$ is a positive semi-definite matrix, which is equivalent to
$$G-\lpa\sum_{i=1}^p\tau_i\lambda^{A_i}_{min}+\sum_{j=1}^r\sigma_j\bar\lambda^{B_j}E_n-\lambda^{C}_{max}\rpa E_n$$
is a positive semi-definite matrix, which implies that for every eigenvalue of $G$, it is greater than or equal to $\sum_{i=1}^p\tau_i \lambda^{A_i}_{min}+\sum_{j=1}^r\sigma_j\bar\lambda^{B_j}-\lambda^{C}_{max}$.\hfill\qed
\end{proof}

Based on the above lemma, the following assumption is given for the establishment of solution method.

\begin{Assumption} \label{assmp2}
There is a critical point $\zeta=(\tau,\sigma)$ of $\vP^d(\zeta)$, satisfying $\Delta>0$ where $$\Delta=\sum_{i=1}^p\tau_i\lambda^{A_i}_{min}+\sum_{j=1}^r\sigma_j\bar\lambda^{B_j}-\lambda^{C}_{max}.$$
\end{Assumption}

\begin{lemma}
If $\bar\zeta$ is a stationary point of $\Pi^d(\zeta) $ satisfying Assumption 1, then $\bar{\zeta}\in\mcal{S}_a^+$.
\end{lemma}
\begin{proof}
From Lemma 3, we know if $\lambda_{G}$ is an arbitrary eigenvalue of $G$, it holds that $\lambda_{G}\geq \Delta$. If $\bar\zeta$ is a critical point satisfying Assumption 1, then $\Delta>0$, so for every eigenvalue of $G$, we have $\lambda_{G}\geq \Delta>0$, then $G$ is a positive definite matrix, i.e., $\bar{\zeta}\in\mcal{S}_a^+$.\hfill\qed
\end{proof}

The following lemma is needed here. Its proof is omitted, which is similar to that of Lemma 6 in \cite{Gao-Wu-triality12}.
\begin{lemma}\label{lm:PplusDUD}
Suppose that $ P\in\mbb{R}^{n\times n}$, $ U\in\mbb{R}^{m\times m}$ and $ W\in\mbb{R}^{n\times m}$ are given symmetric matrices with
$$
 P=\begin{bmatrix} P_{11} &  P_{12}\\ P_{21} &  P_{22} \end{bmatrix}\prec 0, ~~
 U=\begin{bmatrix} U_{11} & 0\\0 &  U_{22} \end{bmatrix}\succ 0, \wand
 W=\begin{bmatrix} W_{11} & 0\\0 & 0 \end{bmatrix},
$$
where $ P_{11}$, $ U_{11}$ and $ W_{11}$ are $r\times r$-dimensional matrices, and $ W_{11}$ is nonsingular. Then,
\begin{equation}\label{eq:apl m}
 - W^T P^{-1} W- U^{-1}\preceq0\Leftrightarrow P+ W U W^T\preceq0.
\end{equation}
\end{lemma}

Now, we give the main result of this paper, triality theorem, which illustrates the relationships between the primal and canonical dual problems on global and local solutions under Assumption 1.
\begin{theorem}\label{th:triality}
{\rm(\textbf{Triality Theorem})}  Suppose that $\bar{\zeta}$ is a critical point of $\vP^d(\zeta)$, and $\bar{x}=G(\bar\zeta)^{\inv}f$.
\benum
\item Min-max duality: If $\bar{\zeta}$ is the critical point satisfying Assumption 1, then the canonical min-max duality holds in the form of
\begin{equation}\label{eq:th2minmax}
\vP(\bar{x})   =\min_{x\in\mbb{R}^n} \vP(x)=\max_{\zeta\in\mcal{S}_a^+} \vP^d(\zeta)=\vP^d(\bar{\zeta}).
\end{equation}
\item Double-max duality: If $\bar{\zeta}\in\mcal{S}_a^-$, the double-max duality holds in the form that if $\bar x$ is a local maximizer of $\vP(x)$ or $\bar\zeta$ is a local maximizer of $\vP^d(\zeta)$,
 we have
    \begin{equation}\label{eq:th2maxmax}
    \vP(\bar{x})   =\max_{x\in\mcal{X}_0} \vP(x)=\max_{\zeta\in\mcal{S}_0} \vP^d(\zeta)=\vP^d(\bar{\zeta})
    \end{equation}
where $\bar x\in \mcal{X}_0 \subset \mbb{R}^n$ and $\bar \zeta\in \mcal{S}_0 \subset \mcal{S}_a^-$.
\item Double-min duality: If $\bar\zeta\in\calS_a^-$, then the double-min duality holds in the form that when $m=n$, if $\bar x$ is a local minimizer of $\vP(x)$ or $\bar\zeta$ is a local minimizer of $\vP^d(\zeta)$, we have
    \begin{equation}\label{eq:th2minmin}
    \vP(\bar{x})   =\min_{x\in\mcal{X}_0} \vP(x)=\min_{\zeta\in\mcal{S}_0} \vP^d(\zeta)=\vP^d(\bar{\zeta})
    \end{equation}
where $\bar x\in \mcal{X}_0 \subset \mbb{R}^n$ and $\bar \zeta\in \mcal{S}_0 \subset \mcal{S}_a^-$.
\eenum
\end{theorem}
\proof
\benum
\item  Because $\bar\zeta$ is a critical point satisfying Assumption 1, by Lemma 4 it holds $\bar{\zeta}\in\mcal{S}_a^+$, i.e., $G(\bar{\zeta})\succ0$. As $G(\bar{\zeta})\succ0$ and $ H\succ0$, by (\ref{eq:der2Pid}) we know the Hessian of the dual function is negative definitive, i.e. $\nabla^2\vP^d(\zeta)\prec0$,  which implies that $\vP^d(\zeta)$ is strictly concave over $\mcal{S}_a^+$. Hence, we get
\[\label{eq:th2Pidmax}
\vP^d(\bar\zeta)=\max_{\zeta\in\calS_a^+} \vP^d(\zeta).
\]
By the convexity of $V(\xi)$, we have
$V(\xi)-V(\bar \xi)\geq (\xi-\bar \xi)^T\nabla V(\bar \xi)=(\xi-\bar \xi)^T\bar{\zeta}$
(see \cite{gao-strang1989}),
 so $$V(\Lambda(x))-V(\Lambda(\bar x))\geq (\Lambda(x)-\Lambda(\bar x))^T\bar{\zeta},$$
which implies (see page 480 \cite{gao-opt03})
 \begin{eqnarray}
\vP(x)-\vP(\bar{x}) & \geq & (\Lambda(x)-\Lambda(\bar x))^T\bar{\zeta}-\frac{1}{2}x^T C x+\frac{1}{2}\bar x^T C \bar x -
f^T(x-\bar x) \nonumber\\
&=& \frac{1}{2}(x -\bar x )^T G(\bar \zeta) (x -\bar x )
+  [G(\bar \zeta) \bar x  - f]^T ( x- \bar x) . \label{1003}
\end{eqnarray}
By the facts that
 $G(\bar \zeta) \bar x  = f$ and
 $G(\bar \zeta)\succ 0$,  we have
   $\vP(x)\geq\vP(\bar{x})$ for any $x\in\mbb{R}^n$, which shows that $\bar x$ is a global minimizer and
  the equation (\ref{eq:th2minmax}) is true   by Theorem \ref{th:AnalSolu}
   and (\ref{eq:th2Pidmax}).

\item If $\bar{\zeta}$ is a local maximizer of $\vP^d(\zeta)$ over $\mcal{S}_a^-$, it is true that $\nabla^2\vP^d(\bar{\zeta})=- Z^T G^\inv Z- H^{-1}\preceq0$ and there exists a neighborhood $\mcal{S}_0\subset\mcal{S}_a^-$ such that for all $\zeta\in\mcal{S}_0$, $\nabla^2\vP^d(\zeta)\preceq0$. Since the map $x= G^{\inv}f$ is continuous over $\calS_a$, the image of the map over $\mcal{S}_0$ is a neighborhood of $\bar x$, which is denoted by $\mcal{X}_0$. Now we prove that for any $x\in\mcal{X}_0$, $\nabla^2\vP(x)\preceq0$, which plus the fact that $\bar{x}$ is a critical point of $\vP(x)$ implies $\bar{x}$ is a maximizer of $\vP(x)$ over $\mcal{X}_0$. By singular value decomposition, there exist orthogonal matrices $ J\in\mbb{R}^{n\times n}$, $ K\in\mbb{R}^{m\times m}$ and $ R\in\mbb{R}^{n\times m}$ with
    \begin{equation}\label{eq:th2R}
     R_{ij}=
    \left\{\begin{array}{ll}
    \delta_i, &~~ i=j \wand i=1,\ldots,r,\\
    0,&~~\wotherwise,
    \end{array}\right.
    \end{equation}
    where $\delta_i>0$ for $i=1,\ldots,r$ and $r=\rank( F)$, such that $Z H^{\frac{1}{2}}= J R K$, then
    \begin{equation}\label{eq:th2FDERK}
     Z = J R K H^{-\frac{1}{2}}.
    \end{equation}
    For any $x\in\mcal{X}_0$, let $\zeta$ be a point satisfying $x= G^{\inv}f$. Therefore, $\nabla^2\vP^d(\zeta)=- Z^T G^\inv Z- H^{-1}\preceq0$, then it holds that
    \begin{equation}\label{eq:th2hessPid}
     - H^{-\frac{1}{2}} K^T R^T J^T G^\inv J R K H^{-\frac{1}{2}}-H^{-1}\preceq0.
    \end{equation}
    Multiplying above inequality by $ K H^{\frac{1}{2}}$ from the left and $ H^{\frac{1}{2}} K^T$ from the right, it can be obtained that
    \begin{equation}\label{eq:th2hessPidequ}
     - R^T J^T G^\inv J R-E_m\preceq0,
    \end{equation}
    which, by Lemma \ref{lm:PplusDUD}, is further equivalent to
    \begin{equation}\label{eq:th2hessPidequfur}
     J^T G J+ R R^T\preceq0,
    \end{equation}
    then it follows that
    \begin{equation}\label{eq:th2hessPi}
    -G\succeq J R R^T J^T=J R K H^{-\frac{1}{2}} H H^{-\frac{1}{2}} K^T R^T J^T= Z H Z^T.
    \end{equation}
    Thus, $\nabla^2\vP(x)= G+ Z H Z^T \preceq 0$, then $\bar{x}$ is a maximizer of $\vP(x)$ over $\calX_0$.

    Similarly, we can prove that if $\bar{x}$ is a maximizer of $\vP(x)$ over $\calX_0$, then $\bar{\zeta}$ is a maximizer of $\vP^d(\zeta)$ over $\mcal{S}_0$. By the Theorem \ref{th:AnalSolu}, the equation (\ref{eq:th2maxmax}) is proved.

\item Now we prove the double-min duality.
Suppose that $\bar{\zeta}$ is a local minimizer of $\vP^d(\zeta)$ in $\mcal{S}_a^-$, then there exists a neighborhood $\mcal{S}_0\subset\mcal{S}_a^-$ of $\bar{\zeta}$ such that for any $\zeta\in\mcal{S}_0$, $\nabla^2\vP^d(\zeta)\succeq0$. Let $\mcal{X}_0$ denote the image of the map $x= G^{\inv}f$ over $\mcal{S}_0$, which is a neighborhood of $\bar x$.
    For any $x\in\mcal{X}_0$, let $\zeta$ be a point that satisfies $x= G^{\inv}f$. It follows from $\nabla^2\vP^d(\zeta)=- Z^T G^\inv Z- H^{-1}\succeq0$ that $- Z^T G^\inv Z\succeq H^{-1}\succ0$, which implies the matrix $ F$ is invertible. Then it is true that
    \begin{equation}\label{eq:th2iii1hessPid}
    - G^\inv\succeq( Z^T)^{-1} H^{-1} Z^{-1},
    \end{equation}
  which is further equivalent to
    \begin{equation}\label{eq:th2iii1hessPi}
    - G\preceq Z H Z^T.
    \end{equation}
   Thus, $\nabla^2\vP(x)= G+ Z H Z^T\succeq0$ and $x$ is a local minimizer of $\vP(x)$. The converse can be proved similarly. By Theorem \ref{th:AnalSolu}, the equation (\ref{eq:th2minmin}) is then true.
\eenum
The theorem is proved. \hfill\qed

This theroem show that by the canonical min-max duaity theory, the nonconvex d.c. programming  problem $(CDC)$ is equivalent to a concave maximization problem
\beq
({\cal P}^d) : \;\;\; \max \{ \vP^d(\zeta)   | \;\;  \zeta\in\mcal{S}_a^+ \} ,
\eeq
which can be solved by well-developed deterministic methods and algorithms, say \cite{ks,pskz,s-s}.

\section{Examples}
In this section, let $p=r=1$. From the definition of (CDC) problem, $A_1$ is a symmetric matrix, $B_1$ and $C_1$ are two positive definite matrices. According to different cases of $A_1$, following five motivating examples are provided to illustrate the proposed canonical duality method in our paper. By examining the critical points of the dual function, we will show how the dualities in the triality theory are verified by these examples.

 \subsection*{Example 1}
We consider the case that $A_1$ is positive definite. Let $\alpha_1=\beta_1=1$ and
$$
 A_1=\begin{bmatrix} 1.5 &  0\\ 0 &  2 \end{bmatrix}, ~~
 B_1=\begin{bmatrix} 0.5 &  0\\ 0 &  3 \end{bmatrix}, ~~
 C_1=\begin{bmatrix} 1.5 & 0\\ 0 & 1 \end{bmatrix},\wand
 f=\begin{bmatrix} 2 \\ 1 \end{bmatrix},
$$
then the primal problem:
$$\min_{(x,y)\in\bbR^2}\vP(x,y)=\exp\lpa0.75x^2+y^2-1\rpa+0.5\lpa 0.25x^2+1.5y^2-1\rpa^2-0.75x^2-0.5y^2-2x-y.$$
The corresponding canonical dual function is
$$
\vP^d(\tau,\sigma)=-0.5\lpa\frac{4}{1.5\tau+0.5\sigma-1.5}+\frac{1}{2\tau+3\sigma-1}\rpa-\tau\ln(\tau)-0.5\sigma^2-\sigma.
$$
In this problem, $\lambda^{A_1}_{min}=1.5$, $\lambda^{B_1}_{min}=0.5$, $\lambda^{B_1}_{max}=3$, and $\lambda^{C_1}_{max}=1.5$. It is noticed that $(\bar\tau_1,\bar\sigma_1)=(2.01147,-0.223104)$ is a critical point of the dual function $\vP^d(\tau,\sigma)$(see Figure \ref{fig:ex1pid1}). As $\bar\sigma_1<0$, we have $\bar\lambda^{B_1}=\lambda^{B_1}_{max}$ and $$\Delta=\bar\tau_1 \lambda^{A_1}_{min}+\bar\sigma_1\lambda^{B_1}_{max}-\lambda^{C_1}_{max}=0.8479>0,$$ so Assumption 1 is satisfied, then $(\bar\tau_1,\bar\sigma_1)$ is in $\calS_a^+$. By Theorem \ref{th:AnalSolu}, we get $(\bar x_1,\bar y_1)=(1.42283, 0.424878)$. Moreover, we have
\[
&\vP(\bar x_1,\bar y_1)=\vP^d(\bar\tau_1,\bar\sigma_1)=-2.8428,\non
\]
so there is no duality gap, then $(\bar x_1,\bar y_1)$ is the global solution of the primal problem, which demonstrates the min-max duality(see Figure \ref{fig:ex1.1}).

\subsection*{Example 2}
We consider the case that $A_1$ is positive semi-definite. Let $\alpha_1=\beta_1=2$ and
$$
 A_1=\begin{bmatrix} 1 &  0\\ 0 &  0 \end{bmatrix}, ~~
 B_1=\begin{bmatrix} 2 &  0\\ 0 &  1 \end{bmatrix}, ~~
 C_1=\begin{bmatrix} 1 & 0\\ 0 & 3 \end{bmatrix},\wand
 f=\begin{bmatrix} 2 \\ 2 \end{bmatrix},
$$
then the primal problem:
$$\min_{(x,y)\in\bbR^2}\vP(x,y)=\exp\lpa0.5x^2-2\rpa+0.5\lpa x^2+0.5y^2-2\rpa^2-0.5x^2-1.5y^2-2x-2y.$$
The corresponding canonical dual function is
$$
\vP^d(\tau,\sigma)=-0.5\lpa\frac{4}{\tau+2\sigma-1}+\frac{4}{\sigma-3}\rpa-\tau\ln(\tau)-\tau-0.5\sigma^2-2\sigma.
$$
In this problem, $\lambda^{A_1}_{min}=0$, $\lambda^{B_1}_{min}=1$, $\lambda^{B_1}_{max}=2$, and $\lambda^{C_1}_{max}=3$. It is noticed that $(\bar\tau_1,\bar\sigma_1)=(0.142222,3.60283)$ is a critical point of the dual function $\vP^d(\tau,\sigma)$(see Figure \ref{fig:ex2pid1}). As $\bar\sigma_1>0$, we have $\bar\lambda^{B_1}=\lambda^{B_1}_{min}$ and $$\Delta=\bar\tau_1 \lambda^{A_1}_{min}+\bar\sigma_1\lambda^{B_1}_{min}-\lambda^{C_1}_{max}=0.60283>0,$$ so Assumption 1 is satisfied, then $(\bar\tau_1,\bar\sigma_1)$ is in $\calS_a^+$. By Theorem \ref{th:AnalSolu}, we get $(\bar x_1,\bar y_1)=(0.315066, 3.3177)$. Moreover, we have
\[
&\vP(\bar x_1,\bar y_1)=\vP^d(\bar\tau_1,\bar\sigma_1)=-17.1934,\non
\]
so there is no duality gap, then $(\bar x_1,\bar y_1)$ is the global solution of the primal problem, which demonstrates the min-max duality(see Figure \ref{fig:ex1.2}).

For showing the double-max duality of Example 2, we find a local maximum point of $\vP^d(\tau,\sigma)$ in ${S}_a^-$: $(\bar\tau_2,\bar\sigma_2)=(0.151452,-1.68381)$. By Theorem \ref{th:AnalSolu}, we get $(\bar x_2,\bar y_2)=(-0.474364, -0.427002)$. Moreover, we have
\[
&\vP(\bar x_2,\bar y_2)=\vP^d(\bar\tau_2,\bar\sigma_2)=2.98579,\non
\]
and $(\bar x_2,\bar y_2)$ is also a local maximum point of $\vP(x,y)$, which demonstrates the double-max duality(see Figure \ref{fig:ex1.2.max}).

\subsection*{Example 3}
We consider the case that $A_1$ is negative definite. Let $\alpha_1=-4$, $\beta_2=0.5$ and
$$
 A_1=\begin{bmatrix} -1 &  0\\ 0 &  -1.5 \end{bmatrix}, ~~
 B_1=\begin{bmatrix} 2 &  0\\ 0 &  1 \end{bmatrix}, ~~
 C_1=\begin{bmatrix} 2 & 0\\ 0 & 3 \end{bmatrix},\wand
 f=\begin{bmatrix} 5 \\ 2 \end{bmatrix},
$$
then the primal problem:
$$\min_{(x,y)\in\bbR^2}\vP(x,y)=\exp\lpa-0.5x^2-0.75y^2+4\rpa+0.5\lpa x^2+0.5y^2-0.5\rpa^2-x^2-1.5y^2-5x-2y.$$
The corresponding canonical dual function is
$$
\vP^d(\tau,\sigma)=-0.5\lpa\frac{25}{-\tau+2\sigma-2}+\frac{4}{-1.5\tau+\sigma-3}\rpa-\tau\ln(\tau)+5\tau-0.5\sigma^2-0.5\sigma.
$$
In this problem, $\lambda^{A_1}_{min}=-1.5$, $\lambda^{B_1}_{min}=1$, $\lambda^{B_1}_{max}=2$, and $\lambda^{C_1}_{max}=3$. It is noticed that $(\bar\tau_1,\bar\sigma_1)=(0.145563,3.95352)$ is a critical point of the dual function $\vP^d(\tau,\sigma)$(see Figure \ref{fig:ex3pid1}). As $\bar\sigma_1>0$, we have $\bar\lambda^{B_1}=\lambda^{B_1}_{min}$ and $$\Delta=\bar\tau_1 \lambda^{A_1}_{min}+\bar\sigma_1\lambda^{B_1}_{min}-\lambda^{C_1}_{max}=0.7352>0,$$ so Assumption 1 is satisfied, then $(\bar\tau_1,\bar\sigma_1)$ is in $\calS_a^+$. By Theorem \ref{th:AnalSolu}, we get $(\bar x_1,\bar y_1)=(0.867833, 2.72044)$. Moreover, we have
\[
&\vP(\bar x_1,\bar y_1)=\vP^d(\bar\tau_1,\bar\sigma_1)=-13.6736,\non
\]
so there is no duality gap, then $(\bar x_1,\bar y_1)$ is the global solution of the primal problem, which demonstrates the min-max duality(see Figure \ref{fig:ex1.3}).

For showing the double-max duality of Example 3, we find a local maximum point of $\vP^d(\tau,\sigma)$ in ${S}_a^-$: $(\bar\tau_2,\bar\sigma_2)=(54.3685,-0.492123)$. By Theorem \ref{th:AnalSolu}, we get $(\bar x_2,\bar y_2)=(-0.0871798, -0.023517)$. Moreover, we have
\[
&\vP(\bar x_2,\bar y_2)=\vP^d(\bar\tau_2,\bar\sigma_2)=54.9641,\non
\]
and $(\bar x_2,\bar y_2)$ is also a a local maximum point of $\vP(x,y)$, which demonstrates the double-max duality(see Figure \ref{fig:ex1.3.max}).

\subsection*{Example 4}
We also consider the case that $A_1$ is indefinite. Let $\alpha_1=1$, $\beta_1=2$ and
$$
 A_1=\begin{bmatrix} -3 &  0\\ 0 &  1 \end{bmatrix}, ~~
 B_1=\begin{bmatrix} 1 &  0\\ 0 &  1 \end{bmatrix}, ~~
 C_1=\begin{bmatrix} 4 & 0\\ 0 & 4.4 \end{bmatrix},\wand
 f=\begin{bmatrix} 1 \\ 1 \end{bmatrix},
$$
then the primal problem:
$$\min_{(x,y)\in\bbR^2}\vP(x,y)=\exp\lpa-1.5x^2+0.5y^2-1\rpa+0.5\lpa 0.5x^2+0.5y^2-2\rpa^2-2x^2-2.2y^2-x-y.$$
The corresponding canonical dual function is
$$
\vP^d(\tau,\sigma)=-0.5\lpa\frac{1}{-3\tau+\sigma-4}+\frac{1}{\tau+\sigma-4.4}\rpa-\tau\ln(\tau)-0.5\sigma^2-2\sigma.
$$
In this problem, $\lambda^{A_1}_{min}=-3$, $\lambda^{B_1}_{min}=\lambda^{B_1}_{max}=1$, and $\lambda^{C_1}_{max}=4.4$. It is noticed that $(\bar\tau_1,\bar\sigma_1)=(0.0612941,4.67004)$ is a critical point of the dual function $\vP^d(\tau,\sigma)$(see Figure \ref{fig:ex5pid1}). As $\bar\sigma_1>0$, we have $\bar\lambda^{B_1}=\lambda^{B_1}_{min}$ and $$\Delta=\bar\tau_1 \lambda^{A_1}_{min}+\bar\sigma_1\lambda^{B_1}_{min}-\lambda^{C_1}_{max}=0.0862>0,$$ so Assumption 1 is satisfied, then $(\bar\tau_1,\bar\sigma_1)$ is in $\calS_a^+$. By Theorem \ref{th:AnalSolu}, we get $(\bar x_1,\bar y_1)=(2.05695, 3.01812)$. Moreover, we have
\[
&\vP(\bar x_1,\bar y_1)=\vP^d(\bar\tau_1,\bar\sigma_1)=-22.6111,\non
\]
so there is no duality gap, then $(\bar x_1,\bar y_1)$ is the global solution of the primal problem, which demonstrates the min-max duality(see Figure \ref{fig:ex1.5}).

For showing the double-max duality of Example 4, we find a local maximum point of $\vP^d(\tau,\sigma)$ in ${S}_a^-$: $(\bar\tau_2,\bar\sigma_2)=(0.361948,-1.97615)$. By Theorem \ref{th:AnalSolu}, we get $(\bar x_2,\bar y_2)=(-0.141603, -0.166273)$. Moreover, we have
\[
&\vP(\bar x_2,\bar y_2)=\vP^d(\bar\tau_2,\bar\sigma_2)=2.52149,\non
\]
and $(\bar x_2,\bar y_2)$ is also a a local maximum point of $\vP(x,y)$, which demonstrates the double-max duality(see Figure \ref{fig:ex1.5.max}).

For showing the double-min duality of Example 4, we find a local minimum point of $\vP^d(\tau,\sigma)$ in ${S}_a^-$: $(\bar\tau_3,\bar\sigma_3)=(0.149286,3.90584)$. By Theorem \ref{th:AnalSolu}, we get $(\bar x_3,\bar y_3)=(-1.84496, -2.89962)$. Moreover, we have
\[
&\vP(\bar x_3,\bar y_3)=\vP^d(\bar\tau_3,\bar\sigma_3)=-12.7833,\non
\]
and $(\bar x_3,\bar y_3)$ is also a a local minimum point of $\vP(x,y)$, which demonstrates the double-min duality(see Figure \ref{fig:ex1.5.min}).

From above double-min duality in Example 4, we can find our proposed canonical dual method can avoids a local minimum point $(\bar x_3,\bar y_3)$ of the primal problem. In fact, by the canonical dual method, the global solution is obtained, so any local minimum point is avoided. For instance, the point $(0.534285,-2.83131)$ is a local minimum point of the primal problem in Example 2(see Figure \ref{fig:3Dex2localmin}), and the local minimum value is -4.78671, but our proposed canonical dual method obtains the global minimum value -17.1934; the point $(1.29672,-2.09209)$ is a local minimum point of the primal problem in Example 3(see Figure \ref{fig:3Dex3localmin}), and the minimum value is -3.98411, but our proposed canonical dual method obtains the global minimum value -13.6736.

\section{Conclusions and further work}\label{se:concl}
Based on the original definition of objectivity in continuum physics,
 a  canonical d.c. optimization problem is proposed, which can be used to model general nonconvex optimization problems in
 complex systems.
Detailed application is provided  by solving a challenging problem in $\bbR^n$.
By the canonical duality theory, this nonconvex problem is able to reformulated as a concave maximization dual problem in a convex domain.
A detailed proof for the  triality theory is provided  under a reasonable assumption. This theory can be used to identify both global and local extrema, and to develop a powerful algorithm for solving this general d.c. optimization problem.
   Several examples are given to illustrate  detailed situations.
All these examples  support the Assumption \ref{assmp2}. However, we should emphasize  that this assumption is only a sufficient condition  for
the existence of a canonical dual solution in     $\mcal{S}_a^+$.
How to relax this assumption  and to obtain a necessary condition  for   $\mcal{S}_a^+ \neq \emptyset$ are still open questions.
We believe that  this condition should be directly related to the coercivity condition (\ref{eq-coe}) of the  target function $\Pi(x)$
 and deserves  detailed study in the future. \\

\noindent{\bf Acknowledgement}:
We are grateful to  anonymous referees and associate editor for their valuable comments and suggestions.
The research was  supported by
 US Air Force Office of Scientific Research under the grant   AFOSR FA9550-10-1-0487. Dr. Jin Zhong was
supported by  National Natural Science Foundation
of China (no. 11401372), Innovation Program of Shanghai Municipal Education Commission (no. 14YZ114) and Science
\& Technology Commission of Shanghai Municipality (no. 12510501700).

\begin{figure}[H]
\centering
\begin{subfigure}[b]{0.25\textwidth}
\centering
\includegraphics[width=0.75\textwidth]{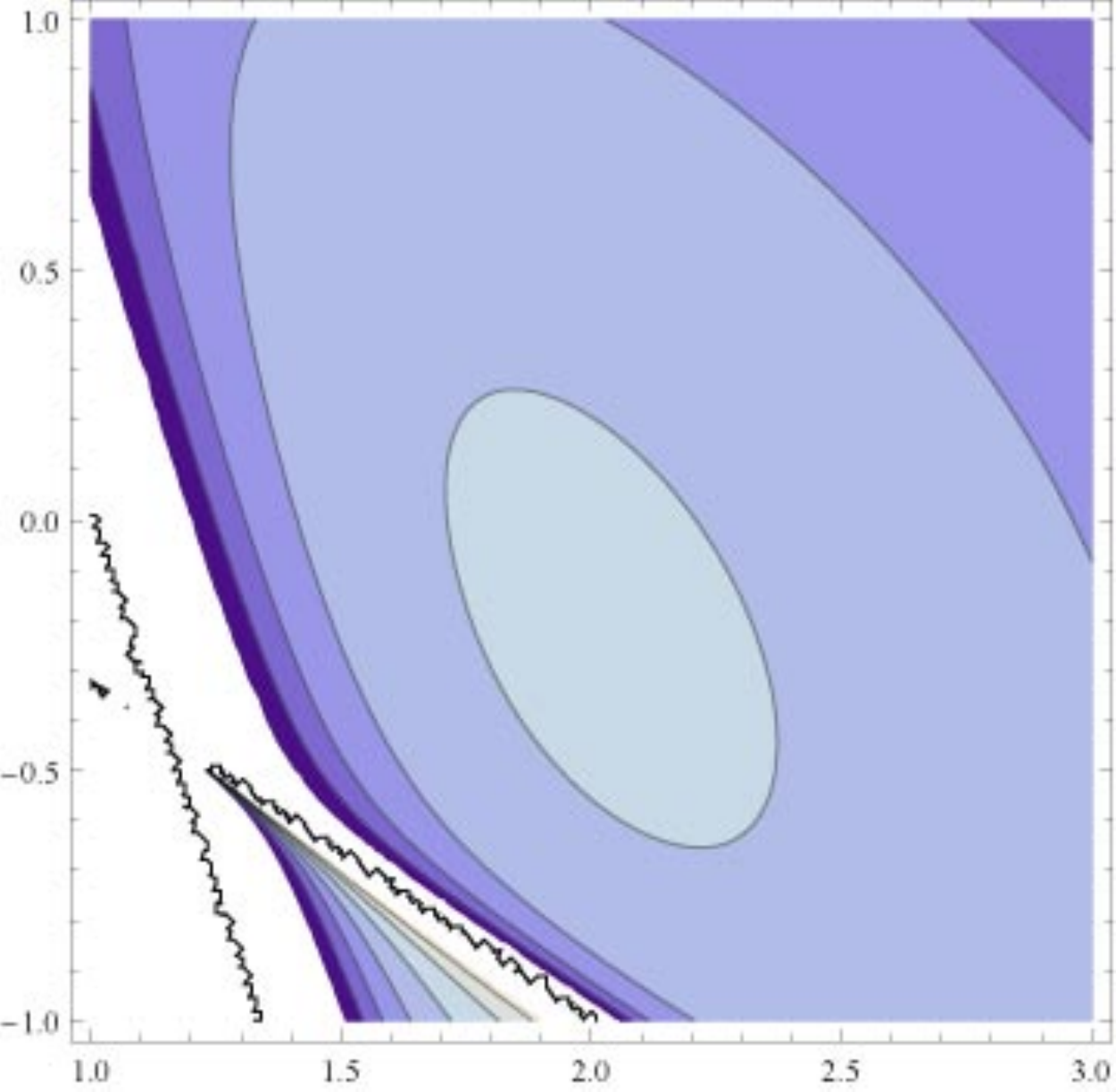}
\caption{}\label{fig:ex1pid1}
\end{subfigure}
\begin{subfigure}[b]{0.25\textwidth}
\centering
\includegraphics[width=0.75\textwidth]{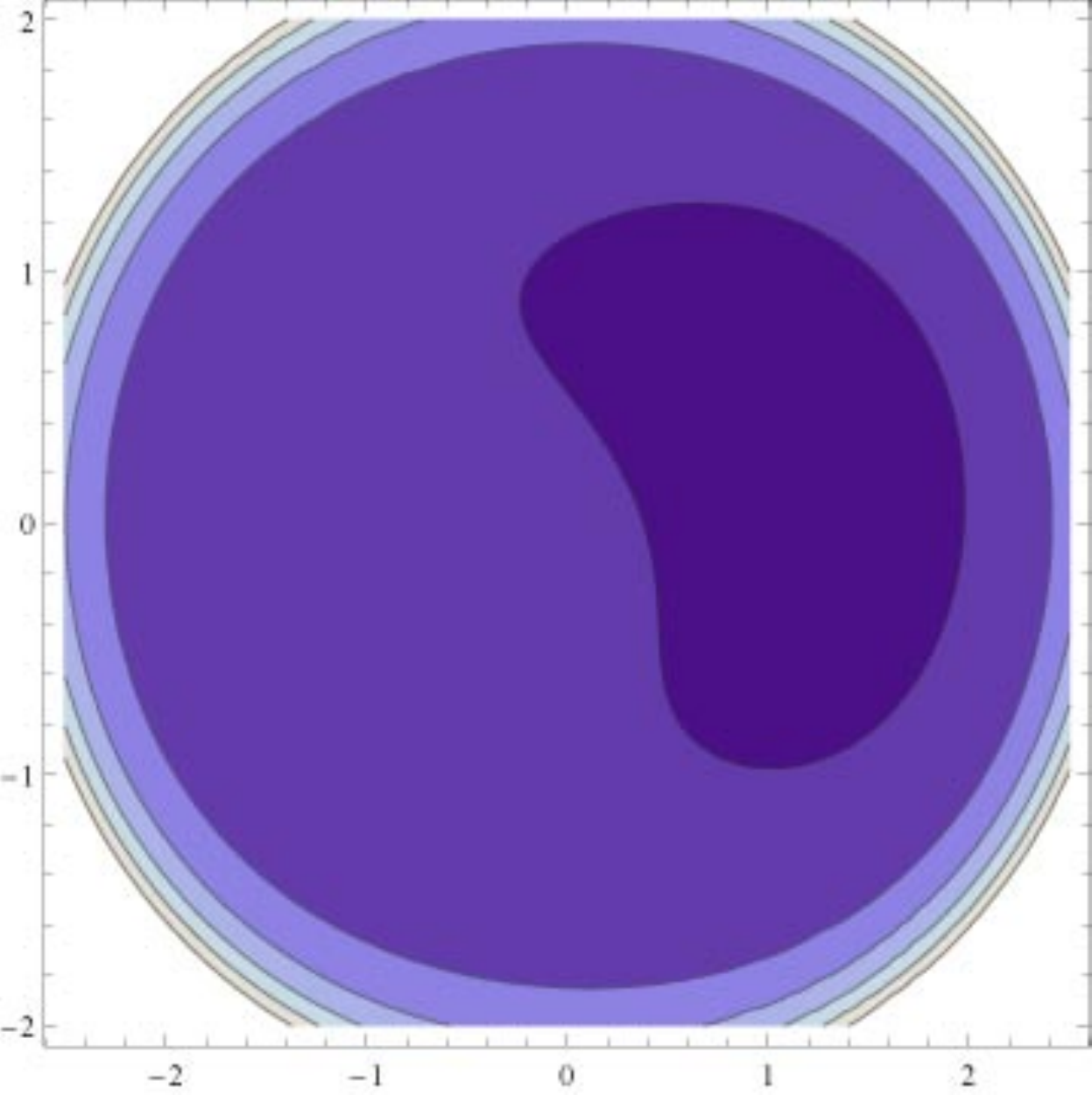}
\caption{}\label{fig:ex1pi1}
\end{subfigure}
\begin{subfigure}[b]{0.4\textwidth}
\centering
\includegraphics[width=1\textwidth]{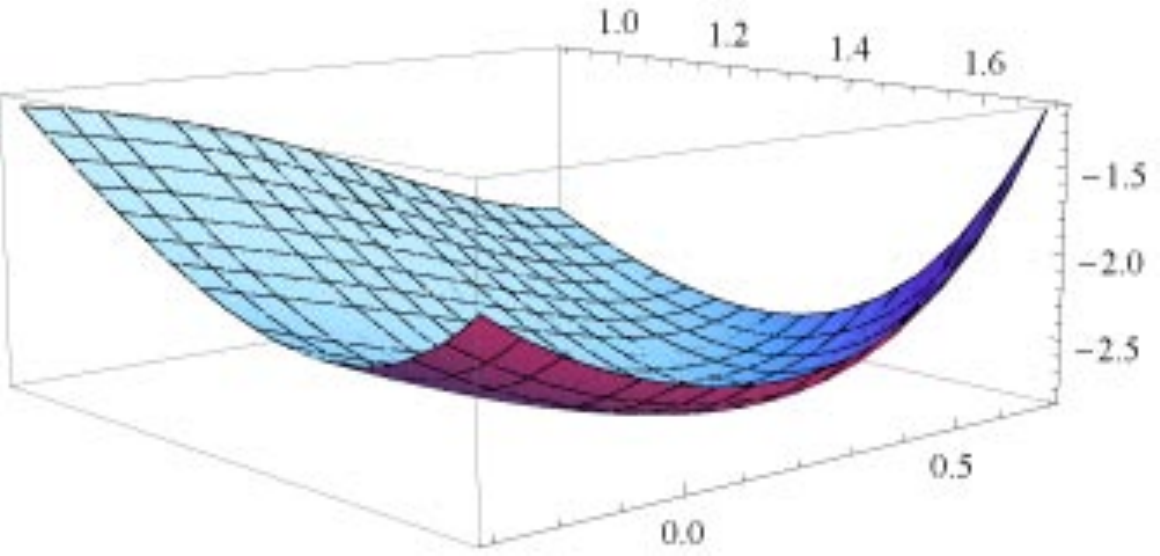}
\caption{}\label{fig:3dexample1pimin}
\end{subfigure}
\caption{The min-max duality in Example 1: (a) contour plot of function $\vP^d(\tau,\sigma)$ near $(\bar\tau_1,\bar\sigma_1)$; (b) contour plot of function $\vP(x,y)$; (c) graph of function $\vP(x,y)$ near $(\bar x_1,\bar y_1)$.}
\label{fig:ex1.1}
\end{figure}

\begin{figure}[H]
\centering
\begin{subfigure}[b]{0.25\textwidth}
\centering
\includegraphics[width=0.75\textwidth]{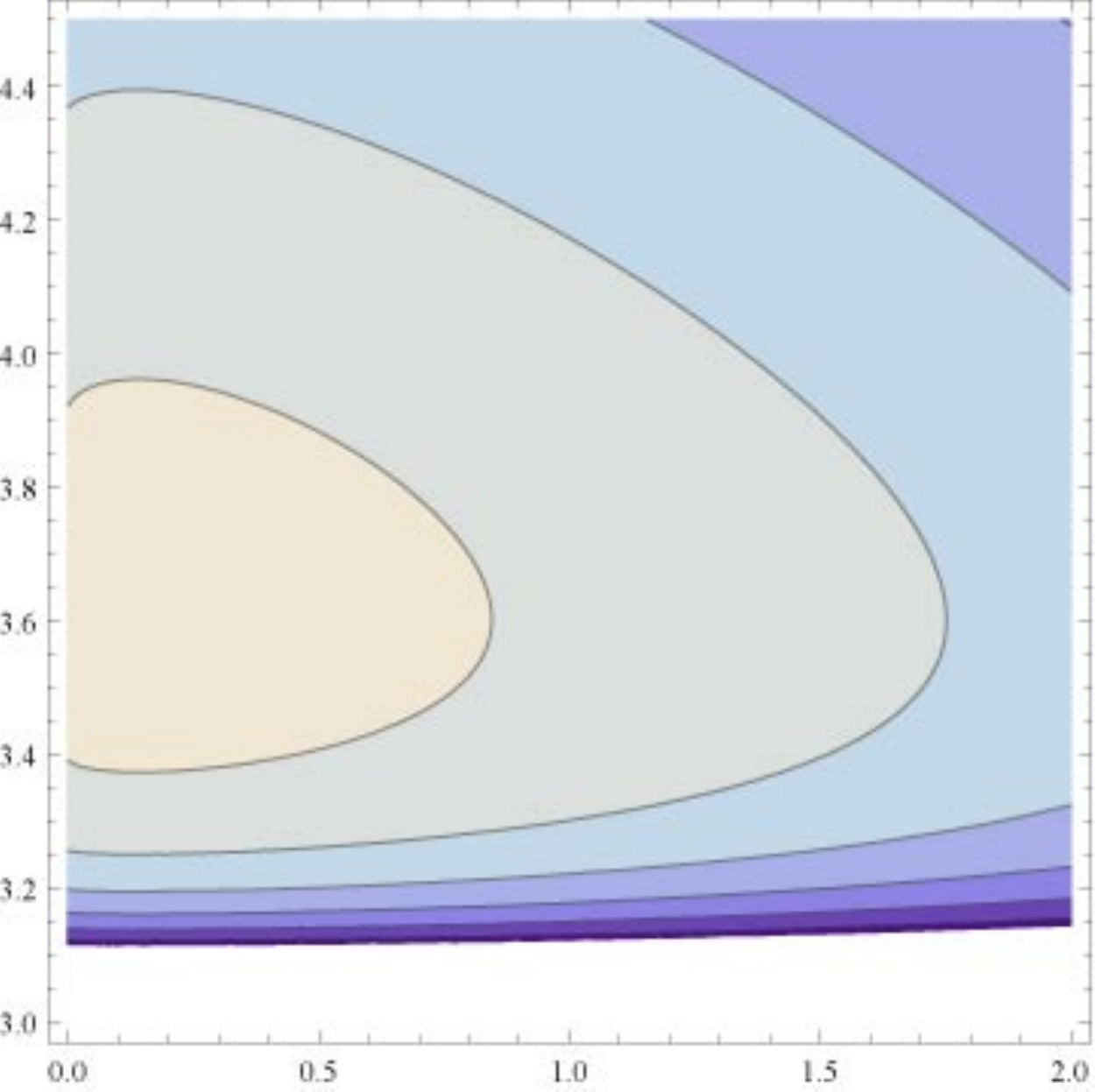}
\caption{}\label{fig:ex2pid1}
\end{subfigure}
\begin{subfigure}[b]{0.25\textwidth}
\centering
\includegraphics[width=0.75\textwidth]{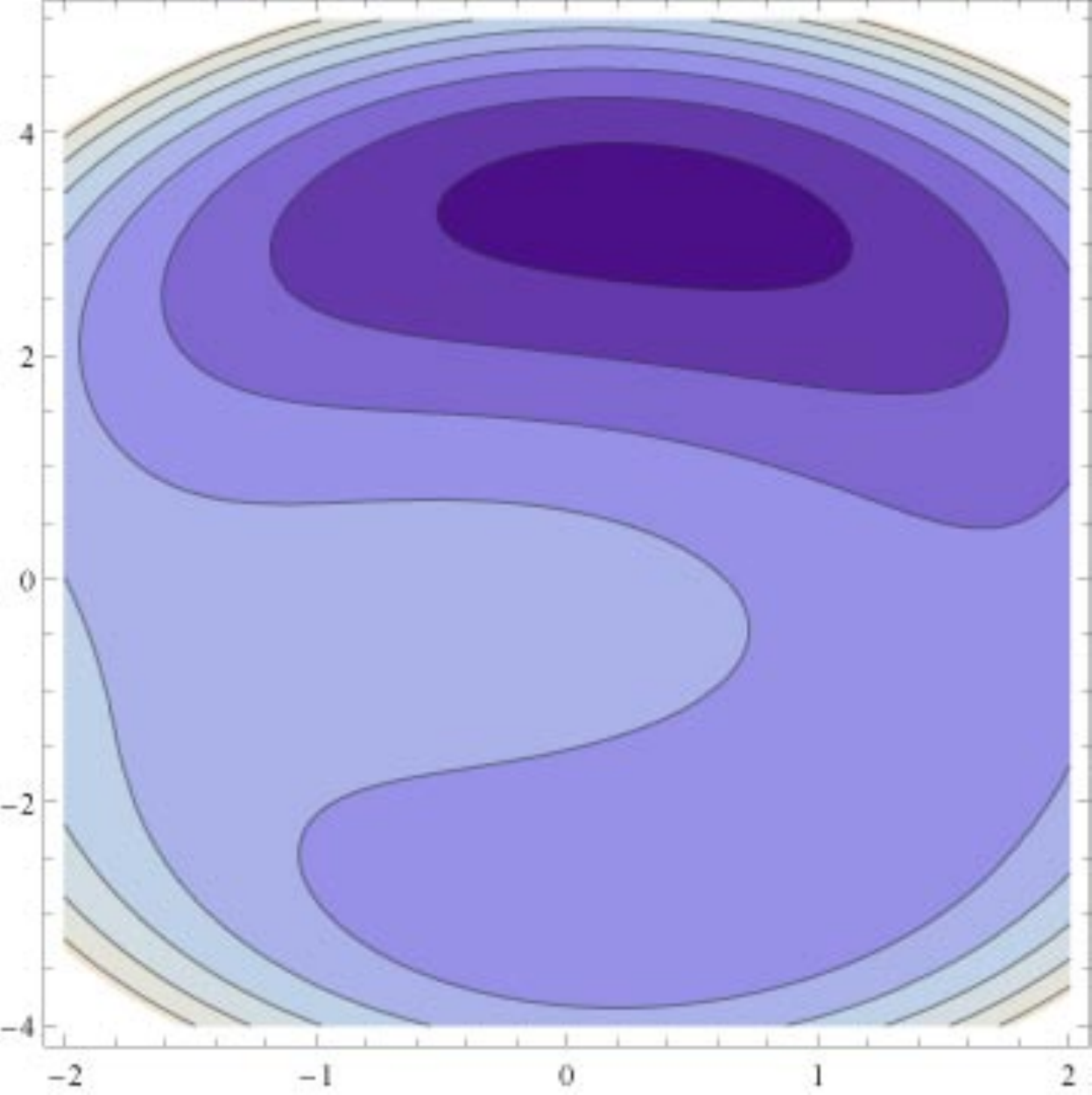}
\caption{}\label{fig:ex2pi2}
\end{subfigure}
\begin{subfigure}[b]{0.4\textwidth}
\centering
\includegraphics[width=1\textwidth]{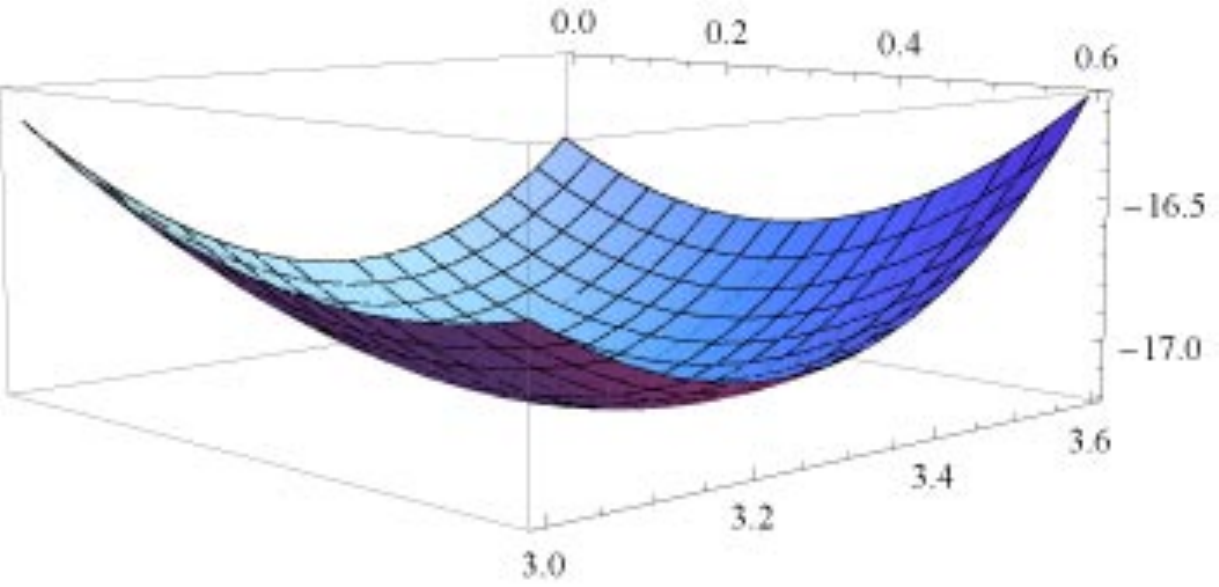}
\caption{}\label{fig:3dexample2pimin}
\end{subfigure}
\caption{The min-max duality in Example 2: (a) contour plot of function $\vP^d(\tau,\sigma)$ near $(\bar\tau_1,\bar\sigma_1)$; (b) contour plot of function $\vP(x,y)$; (c) graph of function $\vP(x,y)$ near $(\bar x_1,\bar y_1)$.}
\label{fig:ex1.2}
\end{figure}

\begin{figure}[H]
\centering
\begin{subfigure}[b]{0.25\textwidth}
\centering
\includegraphics[width=0.75\textwidth]{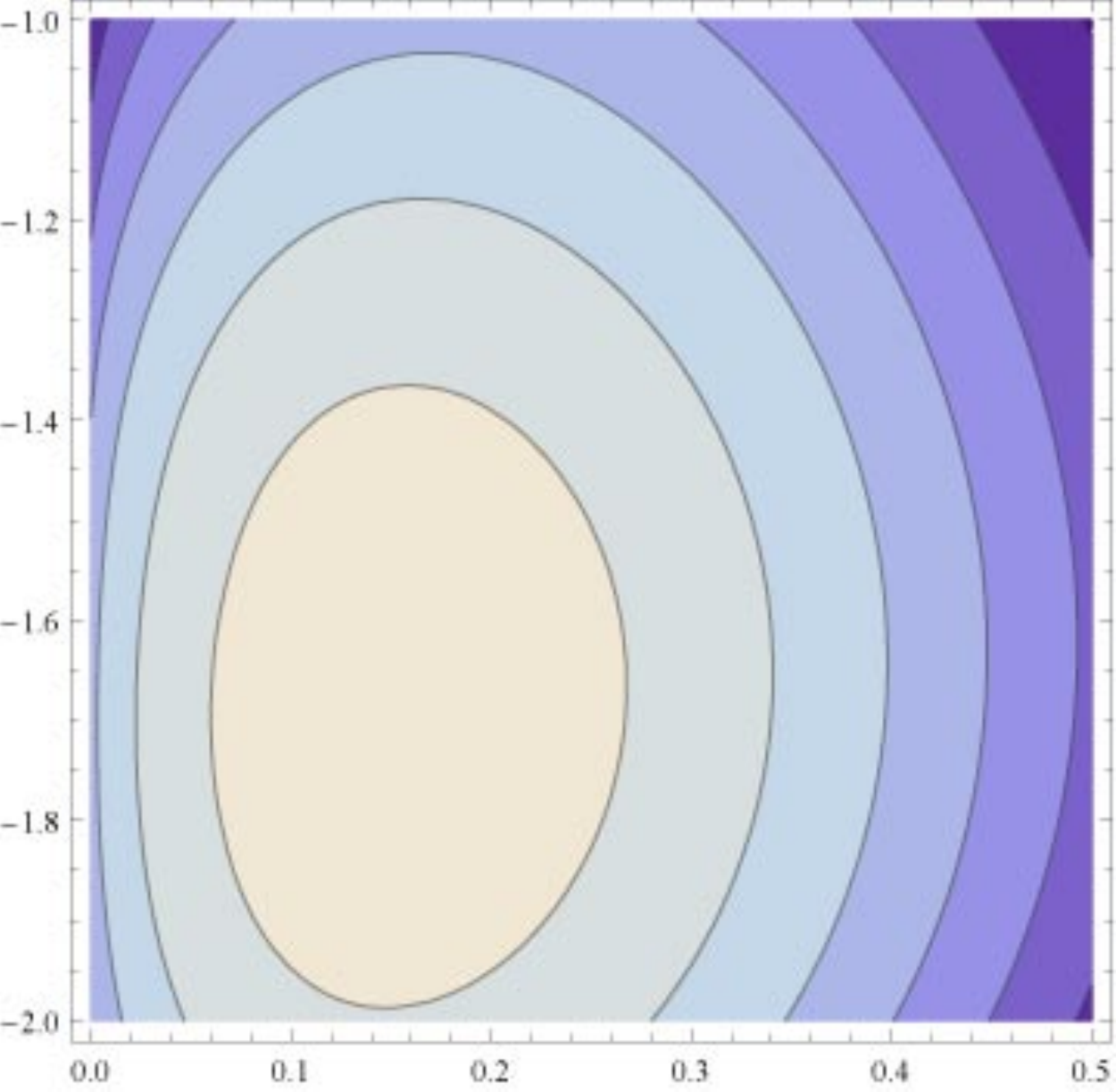}
\caption{}\label{fig:ex2pidmax}
\end{subfigure}
\begin{subfigure}[b]{0.25\textwidth}
\centering
\includegraphics[width=0.75\textwidth]{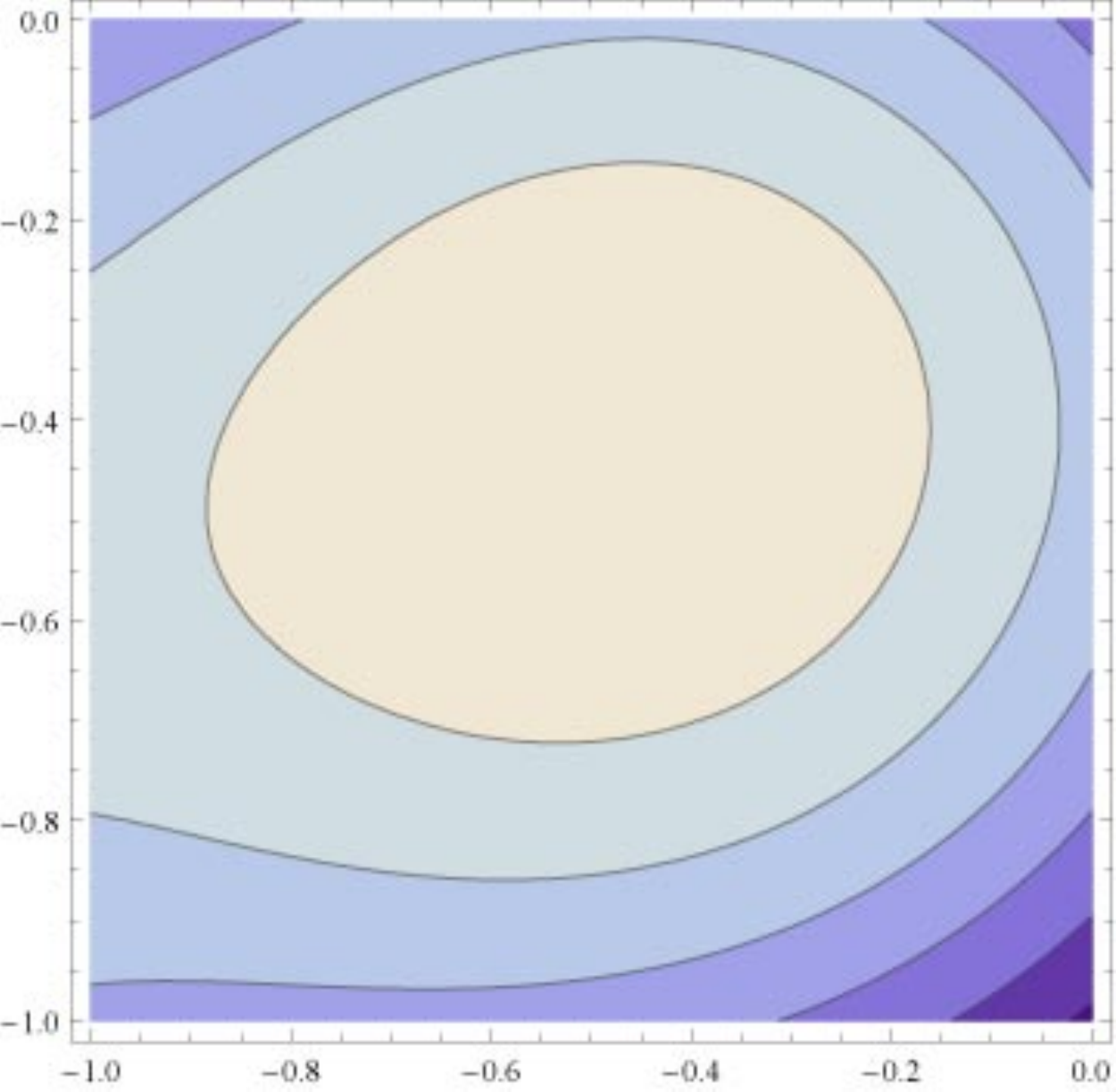}
\caption{}\label{fig:ex2pimax}
\end{subfigure}
\begin{subfigure}[b]{0.4\textwidth}
\centering
\includegraphics[width=0.9\textwidth]{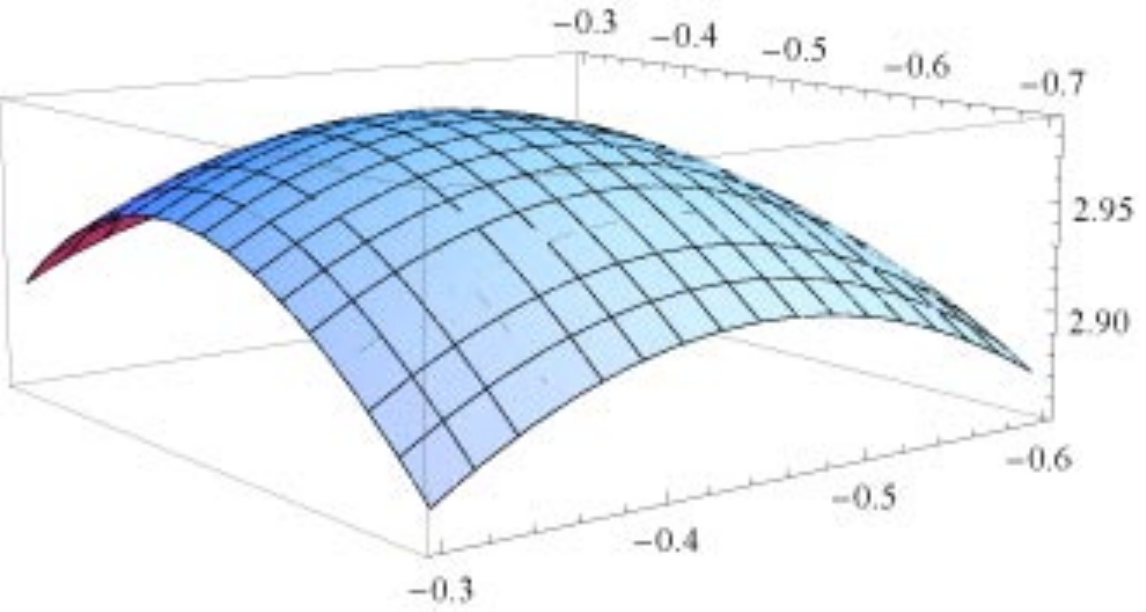}
\caption{}\label{fig:3dexample2pimax}
\end{subfigure}
\caption{The double-max duality in Example 2: (a) contour plot of function $\vP^d(\tau,\sigma)$ near $(\bar\tau_2,\bar\sigma_2)$; (b) contour plot of function $\vP(x,y)$ near $(\bar x_2,\bar y_2)$; (c) graph of function $\vP(x,y)$ near $(\bar x_2,\bar y_2)$.}
\label{fig:ex1.2.max}
\end{figure}

\begin{figure}[H]
\centering
\begin{subfigure}[b]{0.25\textwidth}
\centering
\includegraphics[width=0.75\textwidth]{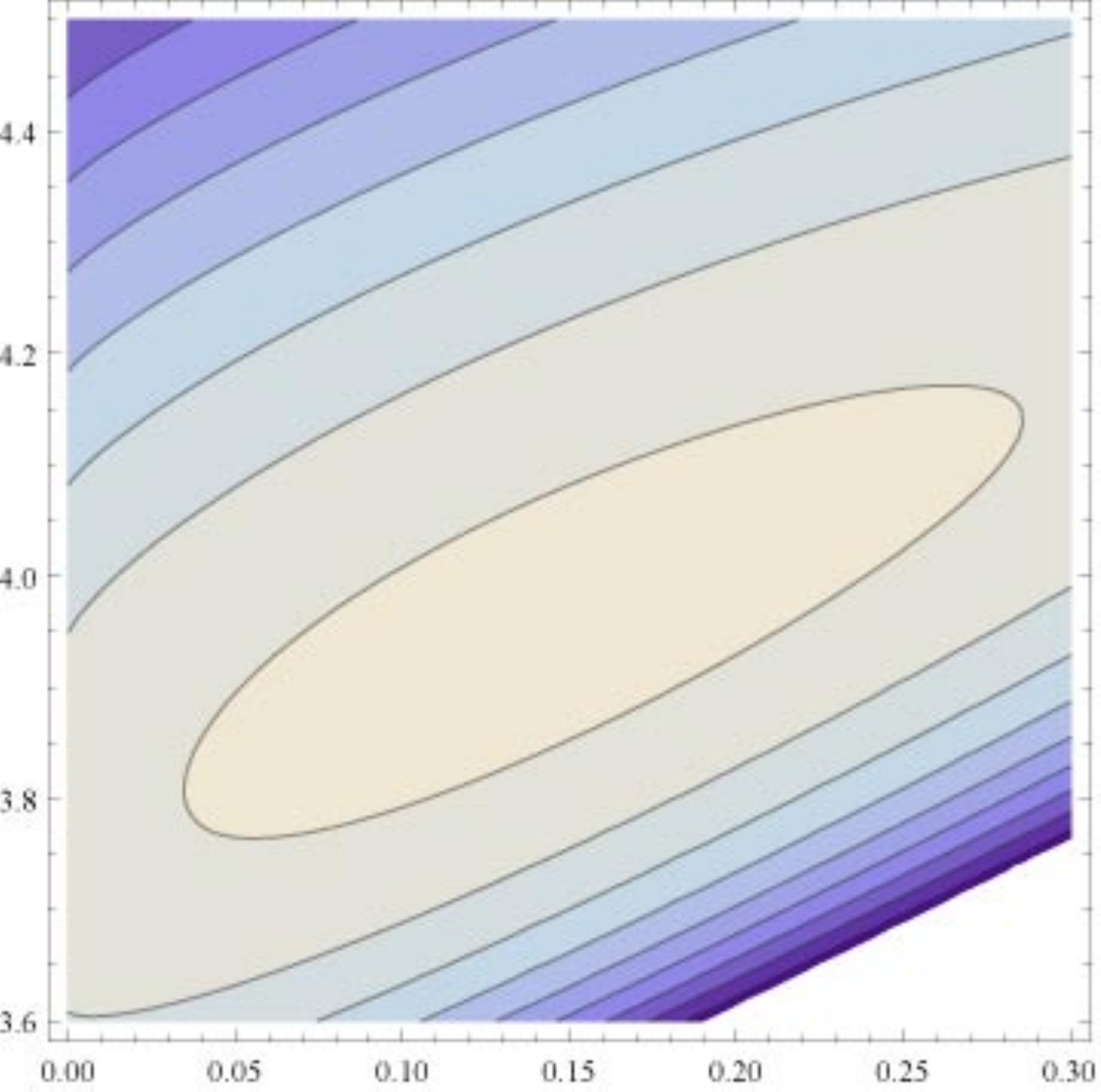}
\caption{}\label{fig:ex3pid1}
\end{subfigure}
\begin{subfigure}[b]{0.25\textwidth}
\centering
\includegraphics[width=0.75\textwidth]{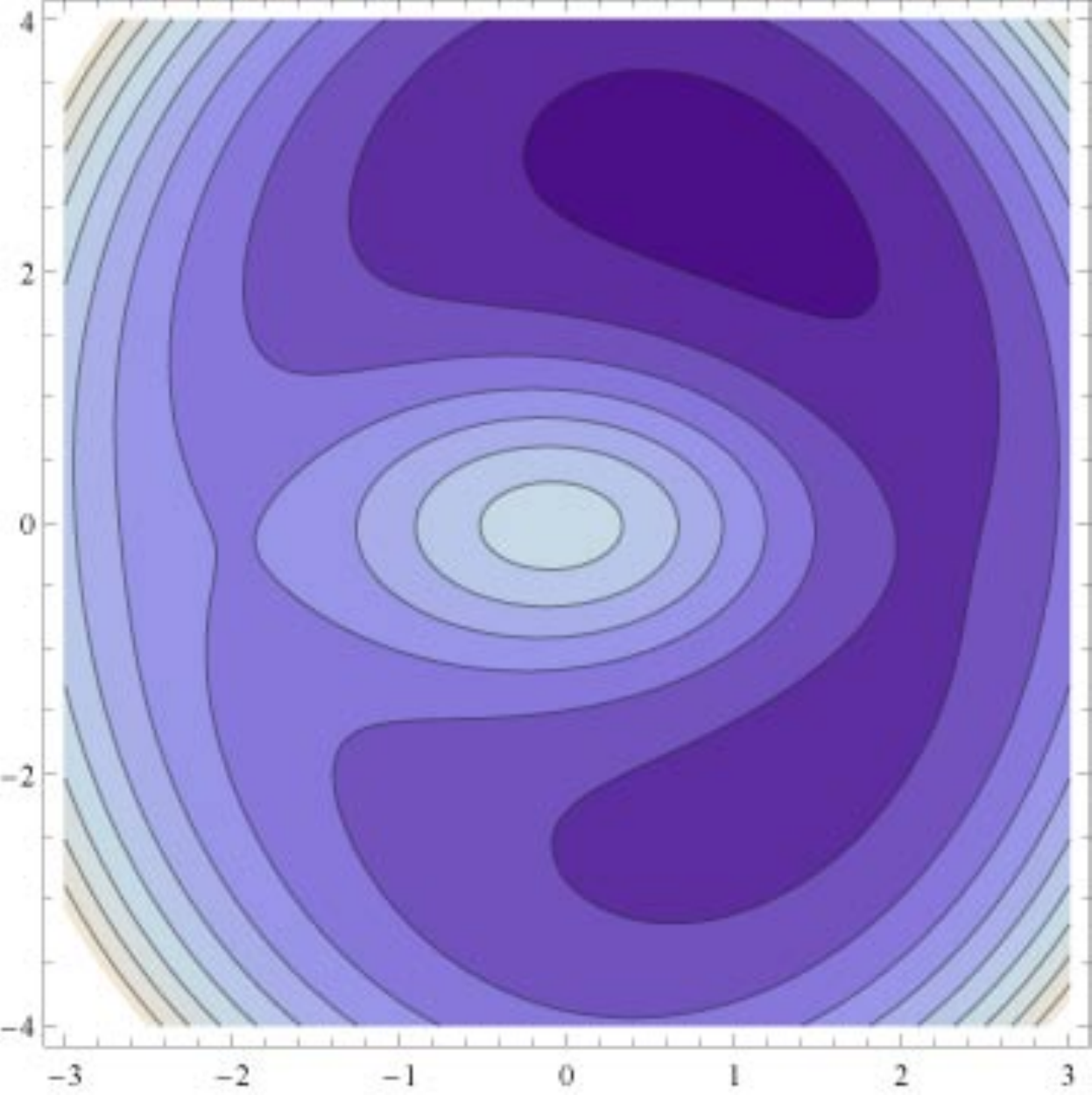}
\caption{}\label{fig:ex3pi2}
\end{subfigure}
\begin{subfigure}[b]{0.4\textwidth}
\centering
\includegraphics[width=1\textwidth]{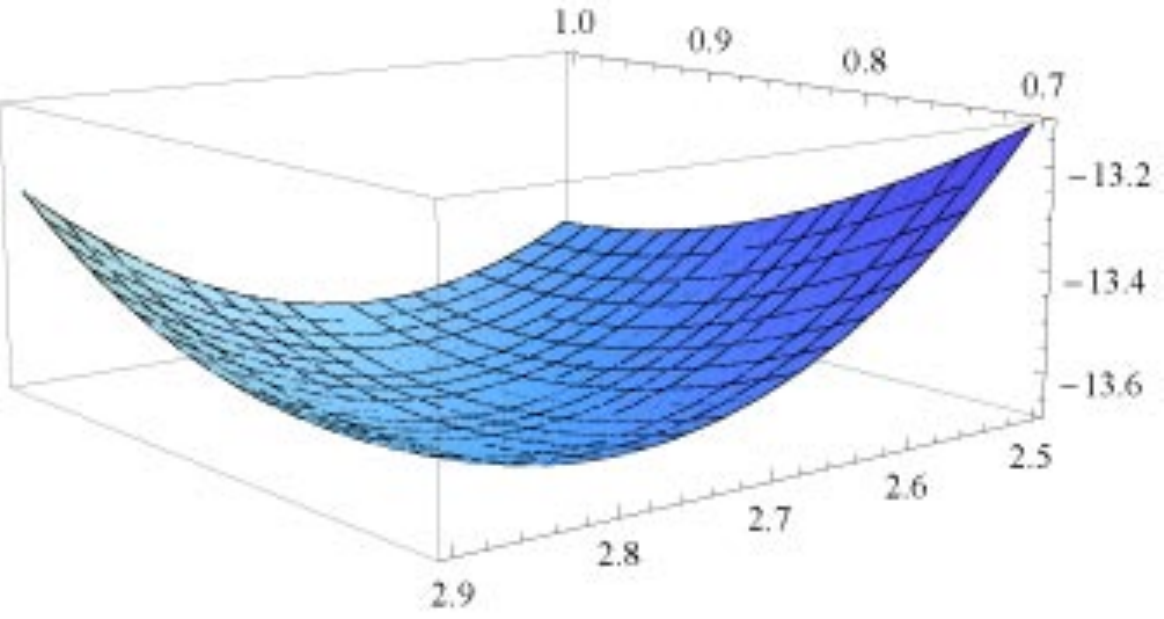}
\caption{}\label{fig:3dexample3pimin}
\end{subfigure}
\caption{The min-max duality in Example 3: (a) contour plot of function $\vP^d(\tau,\sigma)$ near $(\bar\tau_1,\bar\sigma_1)$; (b) contour plot of function $\vP(x,y)$; (c) graph of function $\vP(x,y)$ near $(\bar x_1,\bar y_1)$.}
\label{fig:ex1.3}
\end{figure}

\begin{figure}[H]
\centering
\begin{subfigure}[b]{0.25\textwidth}
\centering
\includegraphics[width=0.75\textwidth]{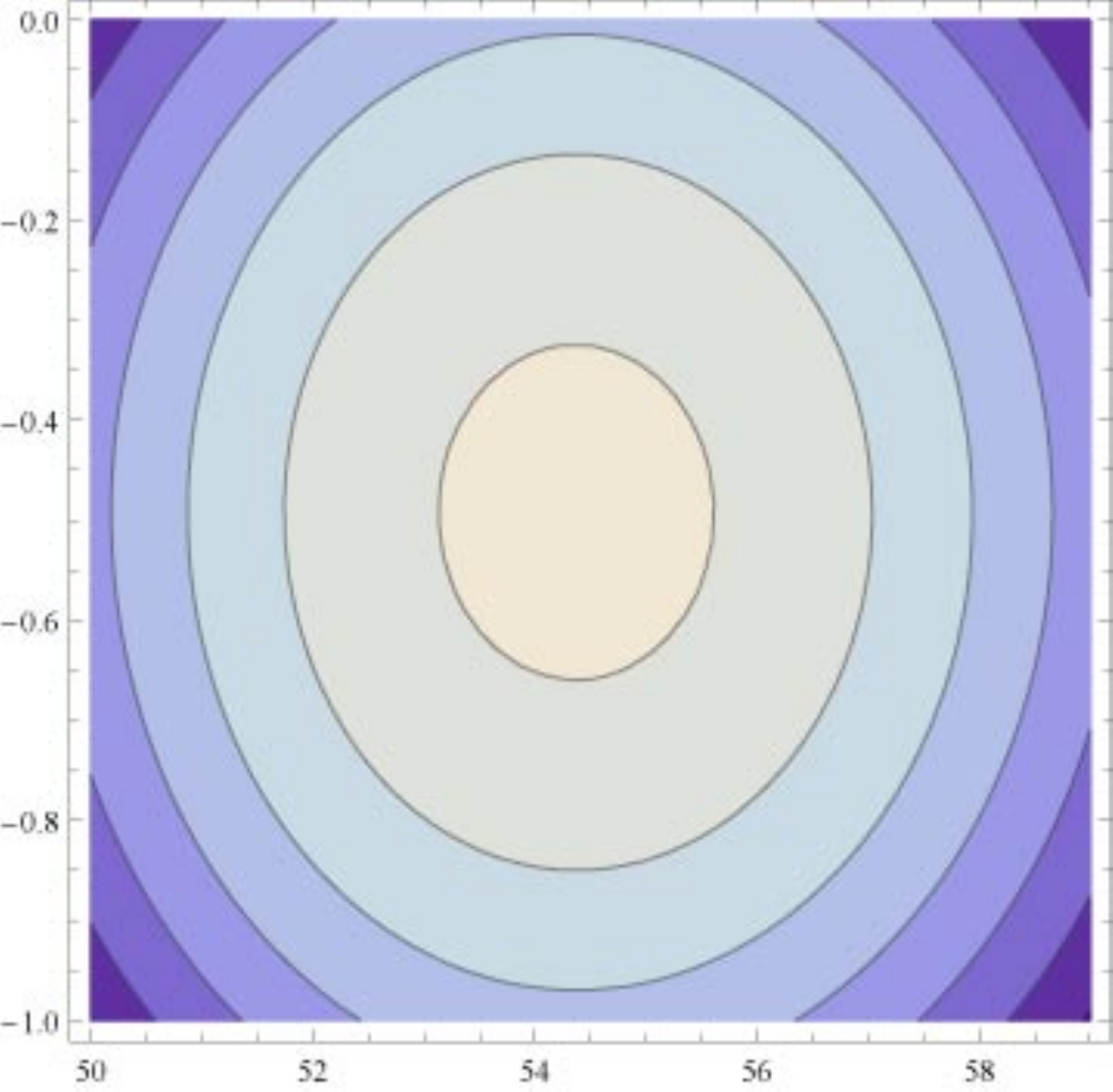}
\caption{}\label{fig:ex3pidmax}
\end{subfigure}
\begin{subfigure}[b]{0.25\textwidth}
\centering
\includegraphics[width=0.75\textwidth]{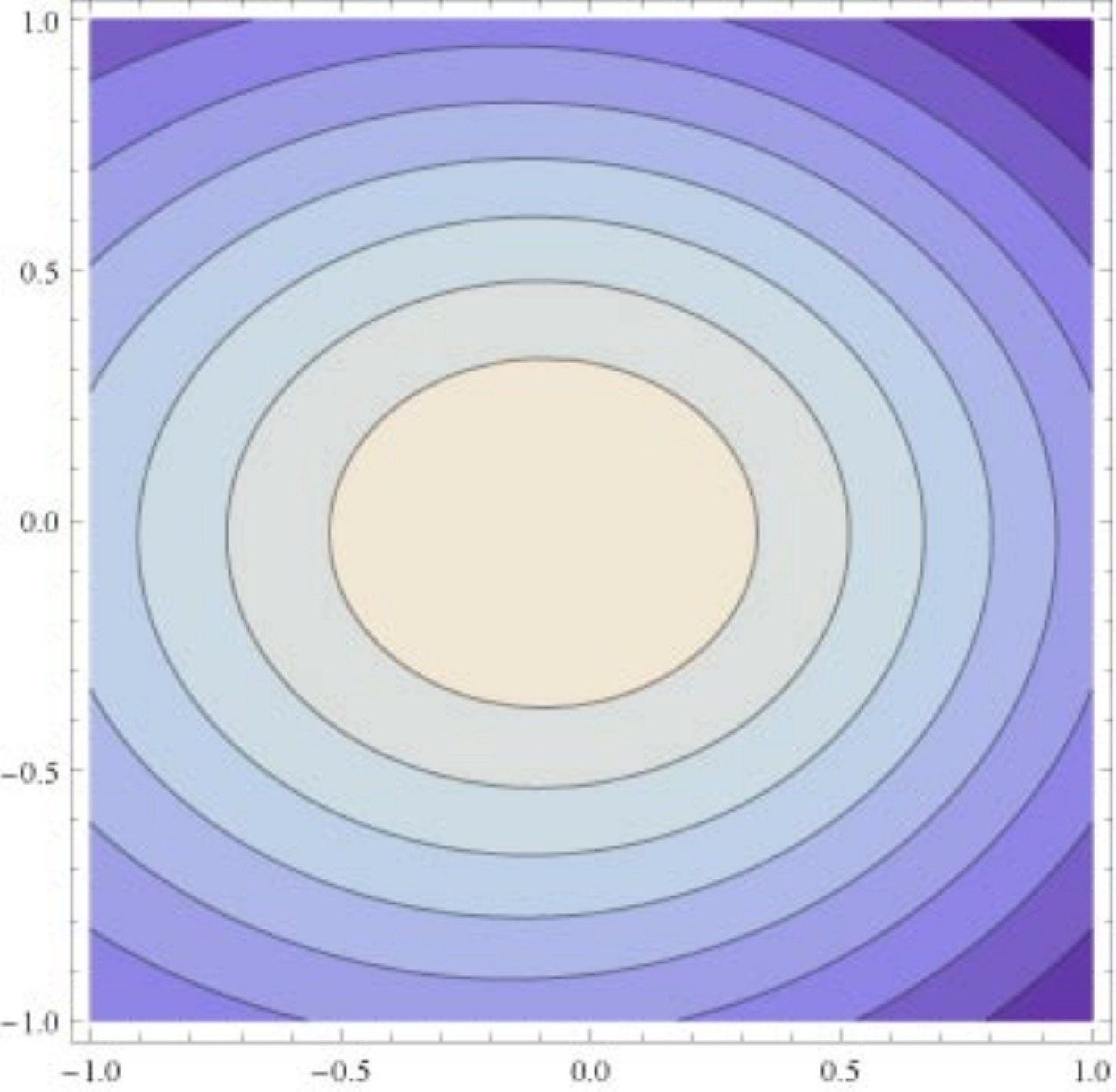}
\caption{}\label{fig:exp3pimax}
\end{subfigure}
\begin{subfigure}[b]{0.4\textwidth}
\centering
\includegraphics[width=0.9\textwidth]{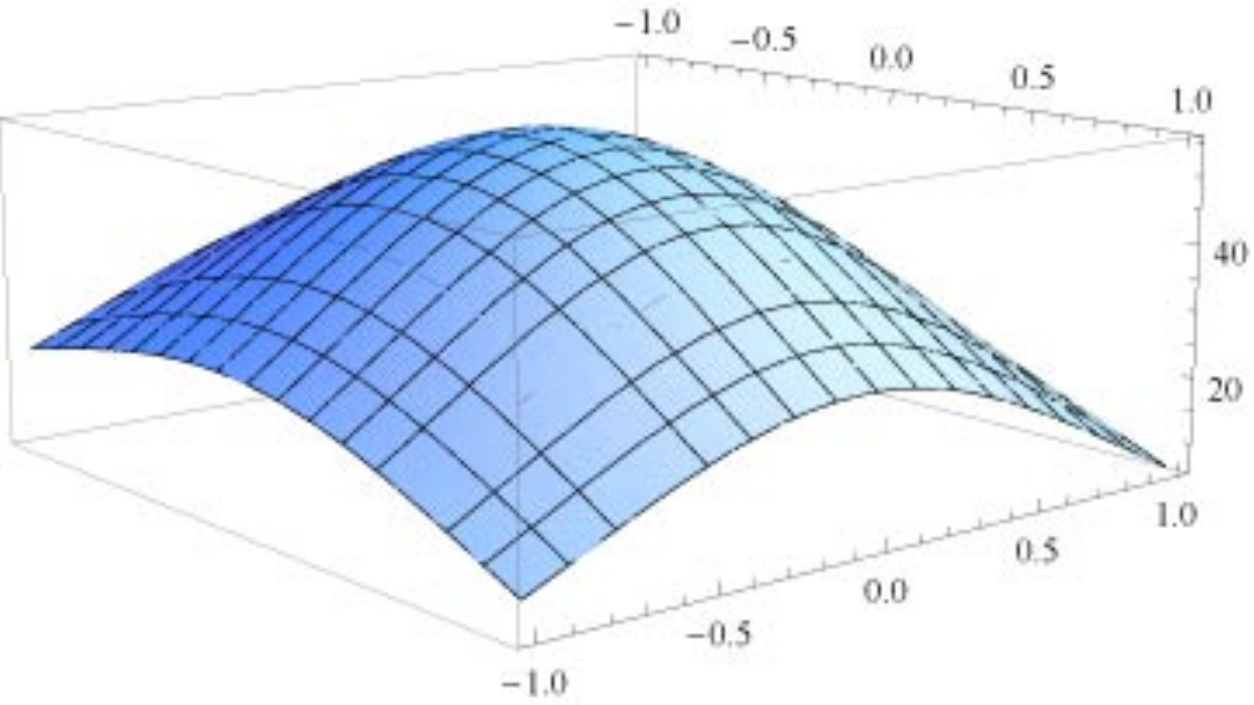}
\caption{}\label{fig:3dexample3pimax}
\end{subfigure}
\caption{The double-max duality in Example 3: (a) contour plot of function $\vP^d(\tau,\sigma)$ near $(\bar\tau_2,\bar\sigma_2)$; (b) contour plot of function $\vP(x,y)$ near $(\bar x_2,\bar y_2)$; (c) graph of function $\vP(x,y)$ near $(\bar x_2,\bar y_2)$.}
\label{fig:ex1.3.max}
\end{figure}

\begin{figure}[H]
\centering
\begin{subfigure}[b]{0.25\textwidth}
\centering
\includegraphics[width=0.75\textwidth]{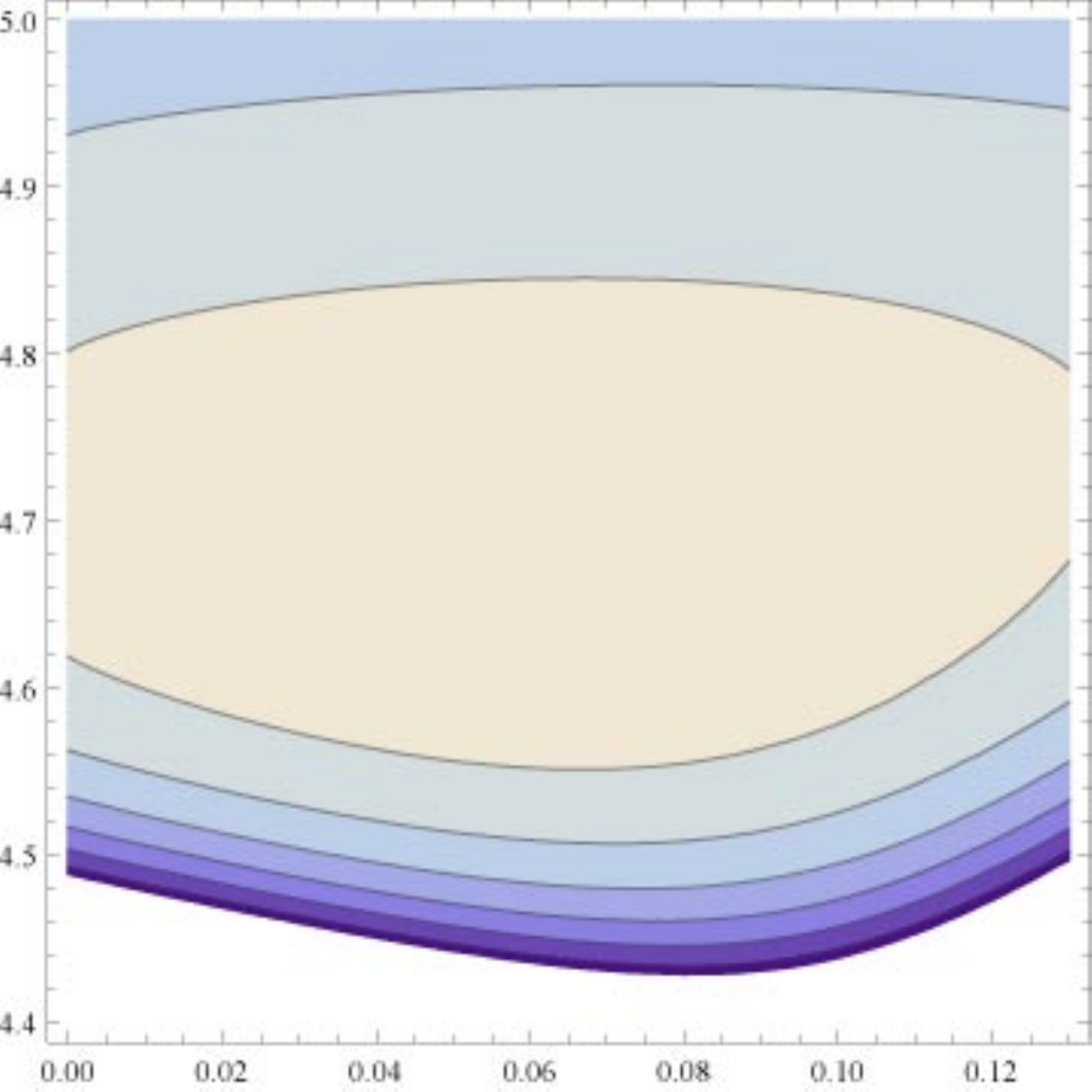}
\caption{}\label{fig:ex5pid1}
\end{subfigure}
\begin{subfigure}[b]{0.25\textwidth}
\centering
\includegraphics[width=0.75\textwidth]{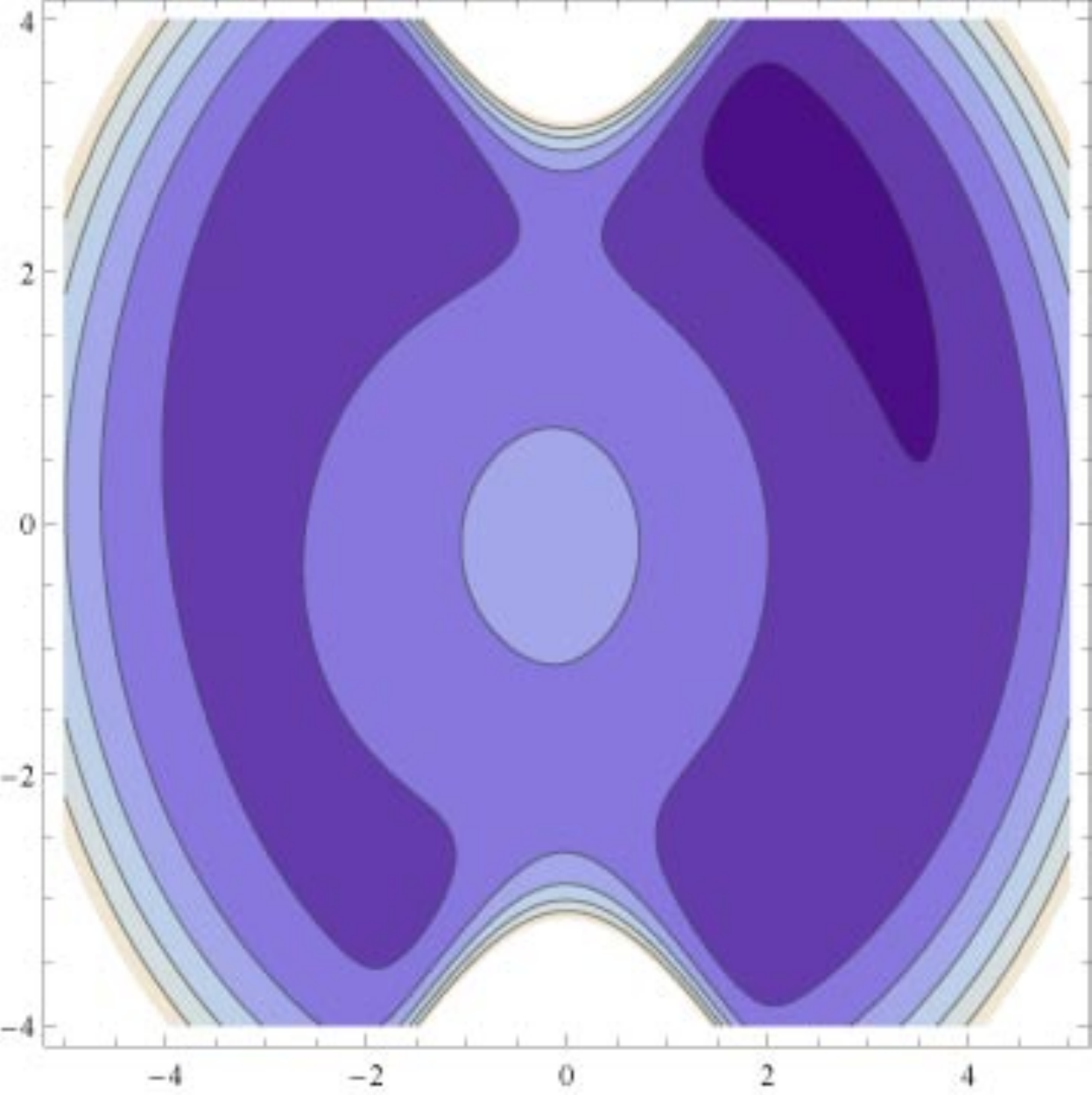}
\caption{}\label{fig:ex5pi2}
\end{subfigure}
\begin{subfigure}[b]{0.4\textwidth}
\centering
\includegraphics[width=0.9\textwidth]{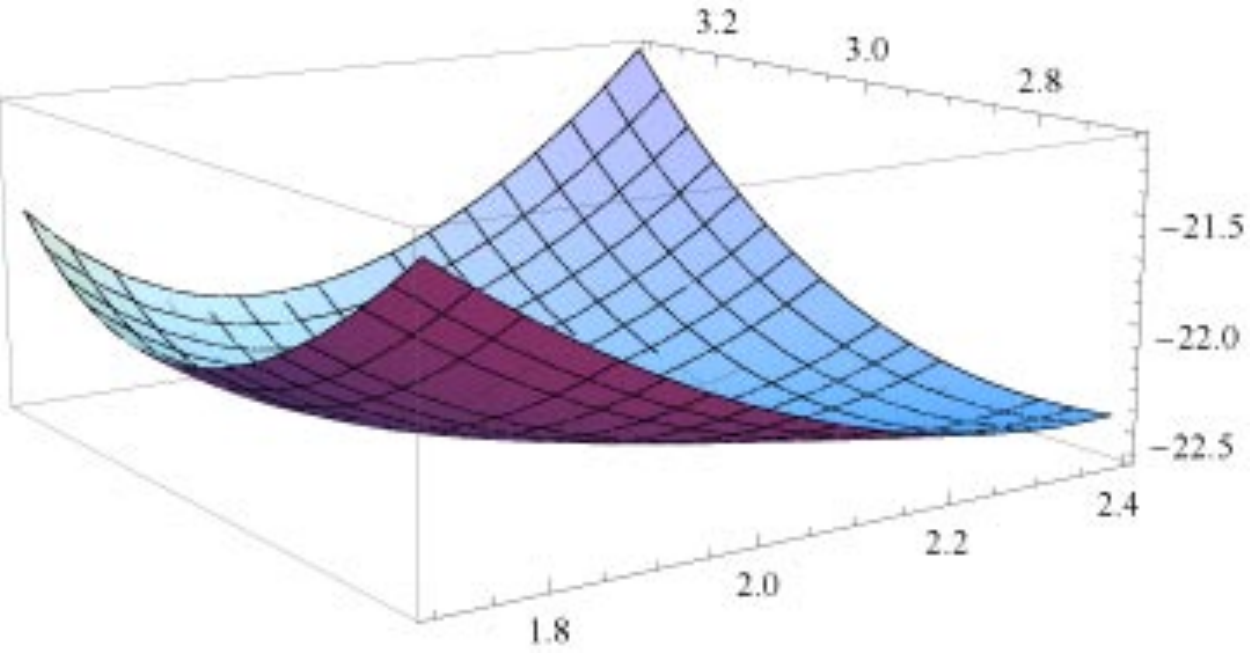}
\caption{}\label{fig:3dexample5pimin}
\end{subfigure}
\caption{The min-max duality in Example 4: (a) contour plot of function $\vP^d(\tau,\sigma)$ near $(\bar\tau_1,\bar\sigma_1)$; (b) contour plot of function $\vP(x,y)$; (c) graph of function $\vP(x,y)$ near $(\bar x_1,\bar y_1)$.}
\label{fig:ex1.5}
\end{figure}

\begin{figure}[H]
\centering
\begin{subfigure}[b]{0.25\textwidth}
\centering
\includegraphics[width=0.75\textwidth]{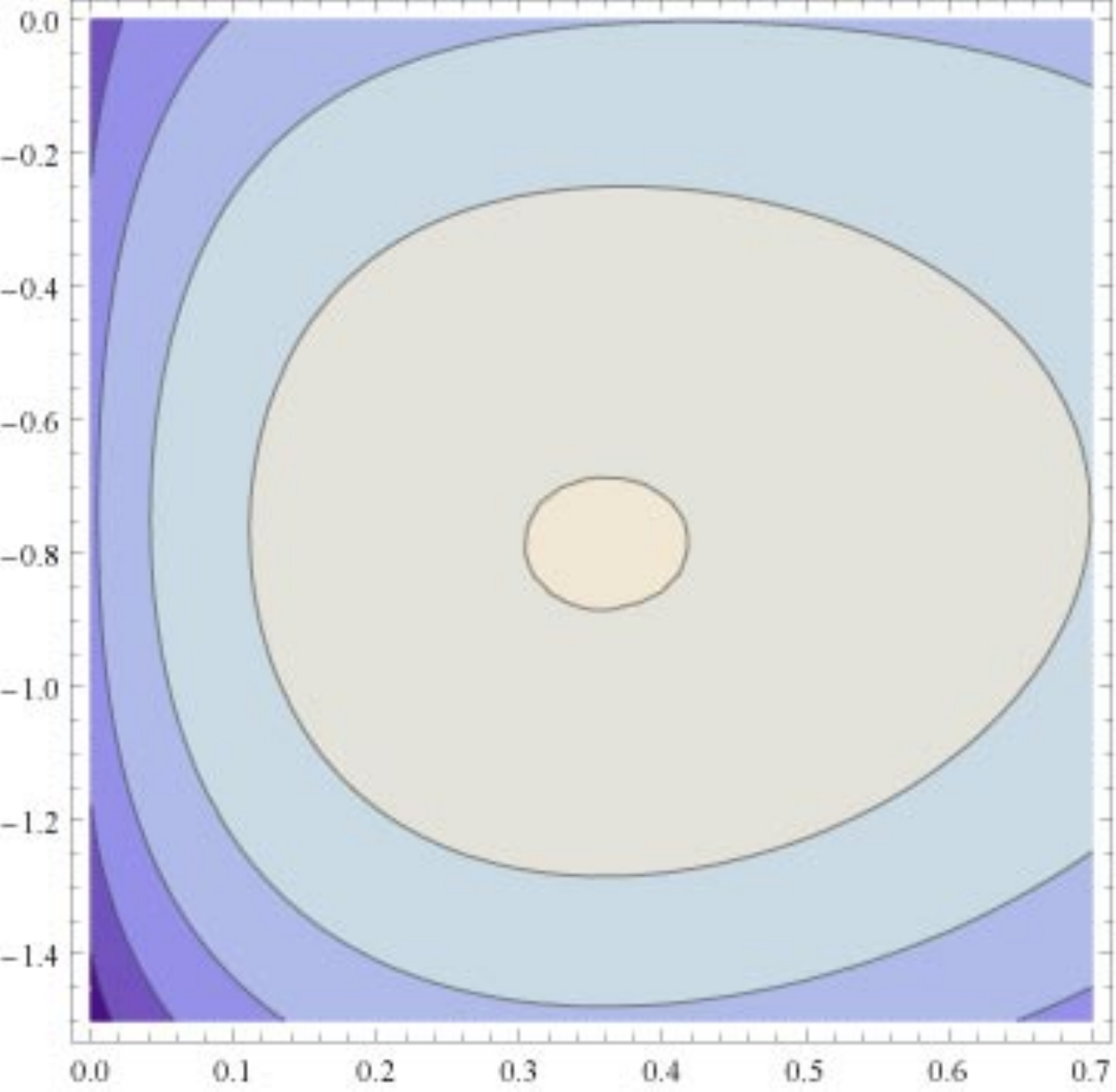}
\caption{}\label{fig:ex5pidmax}
\end{subfigure}
\begin{subfigure}[b]{0.25\textwidth}
\centering
\includegraphics[width=0.75\textwidth]{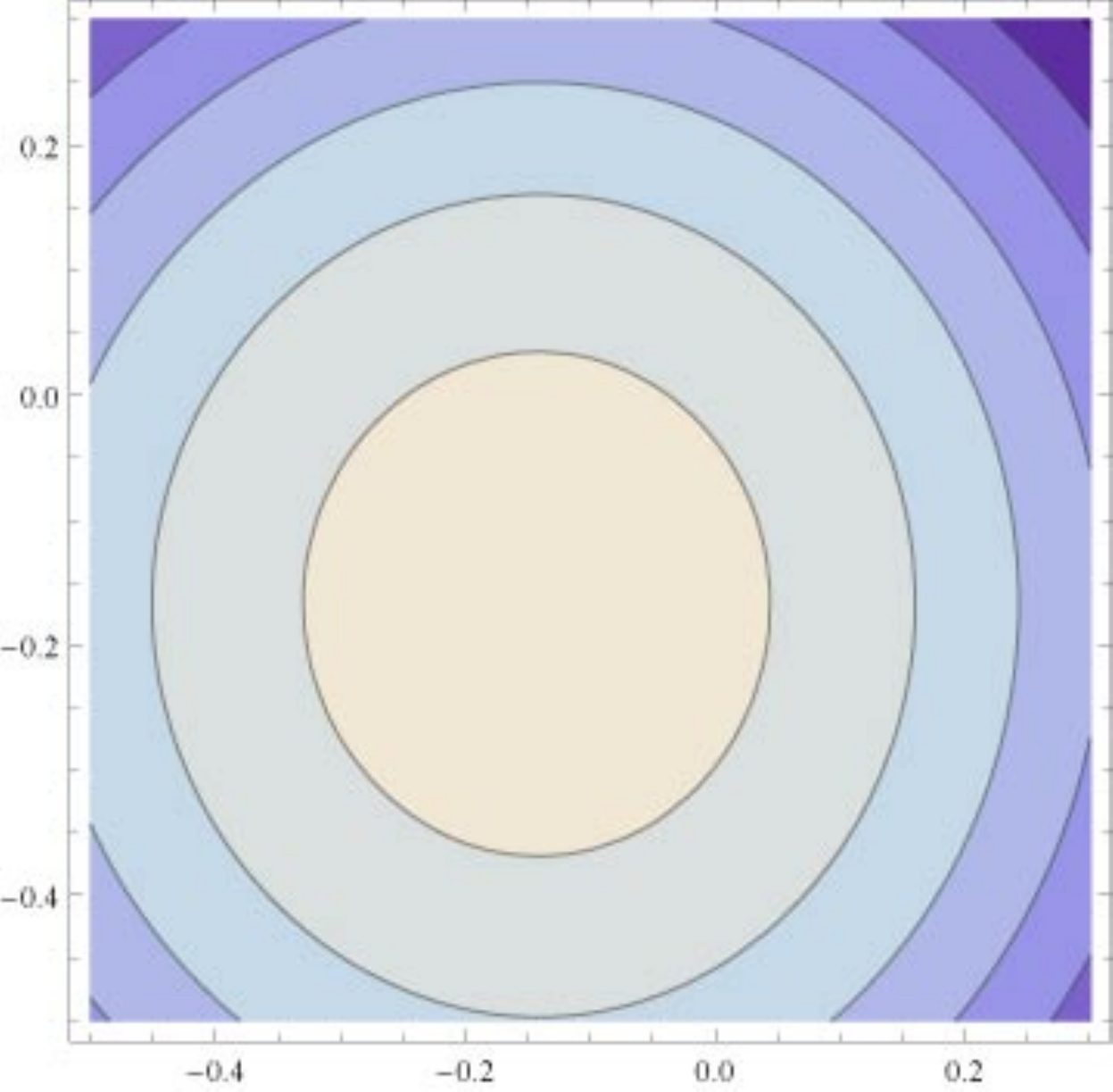}
\caption{}\label{fig:ex5pimax}
\end{subfigure}
\begin{subfigure}[b]{0.4\textwidth}
\centering
\includegraphics[width=0.9\textwidth]{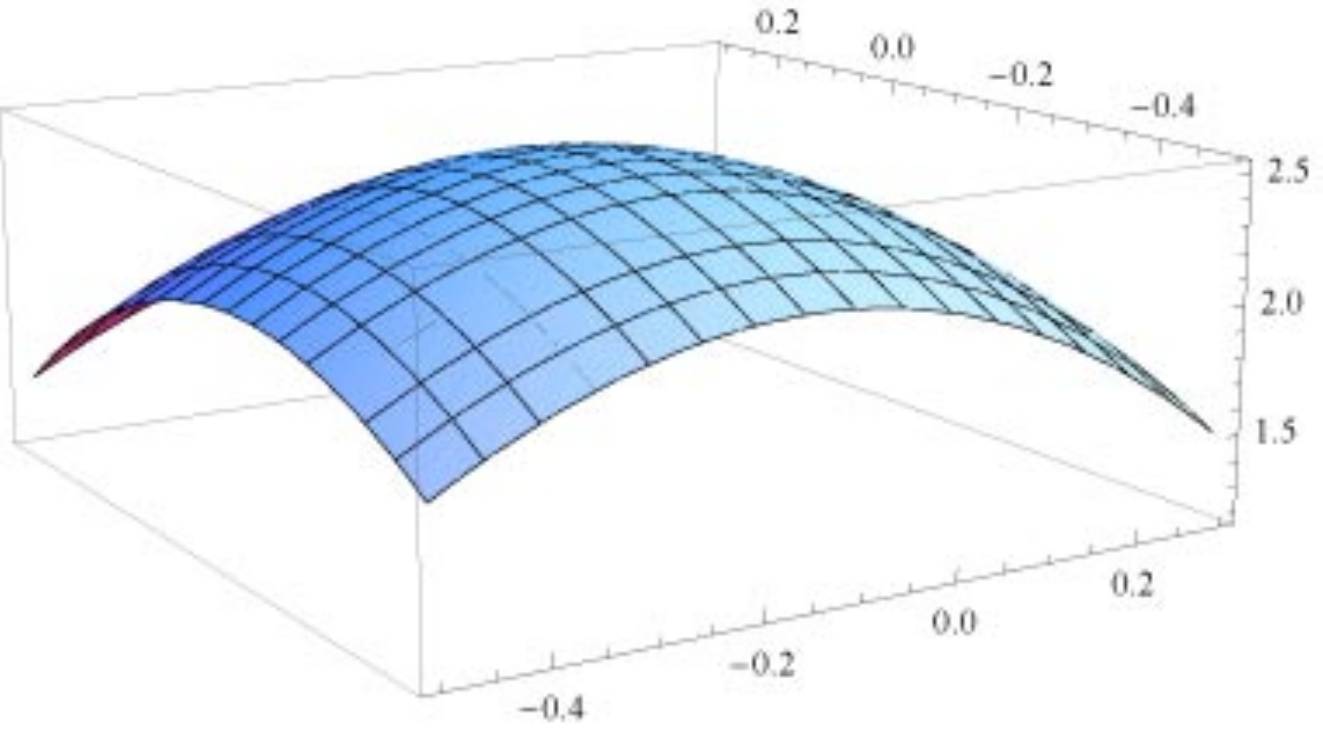}
\caption{}\label{fig:3dexample5pimax}
\end{subfigure}
\caption{The double-max duality in Example 4: (a) contour plot of function $\vP^d(\tau,\sigma)$ near $(\bar\tau_2,\bar\sigma_2)$; (b) contour plot of function $\vP(x,y)$ near $(\bar x_2,\bar y_2)$; (c) graph of function $\vP(x,y)$ near $(\bar x_2,\bar y_2)$.}
\label{fig:ex1.5.max}
\end{figure}

\begin{figure}[H]
\centering
\begin{subfigure}[b]{0.25\textwidth}
\centering
\includegraphics[width=0.75\textwidth]{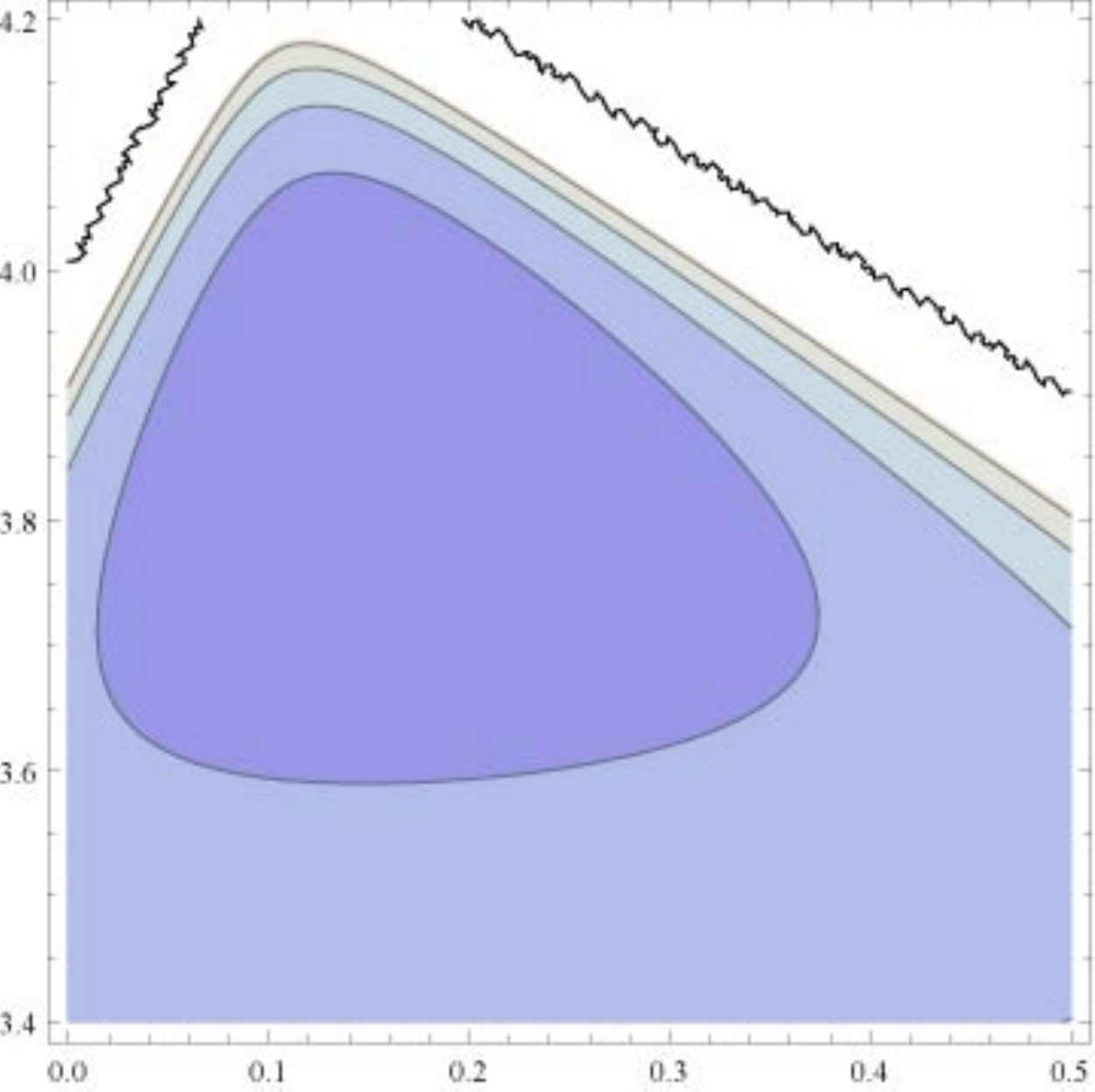}
\caption{}\label{fig:ex5pidmin}
\end{subfigure}
\begin{subfigure}[b]{0.25\textwidth}
\centering
\includegraphics[width=0.75\textwidth]{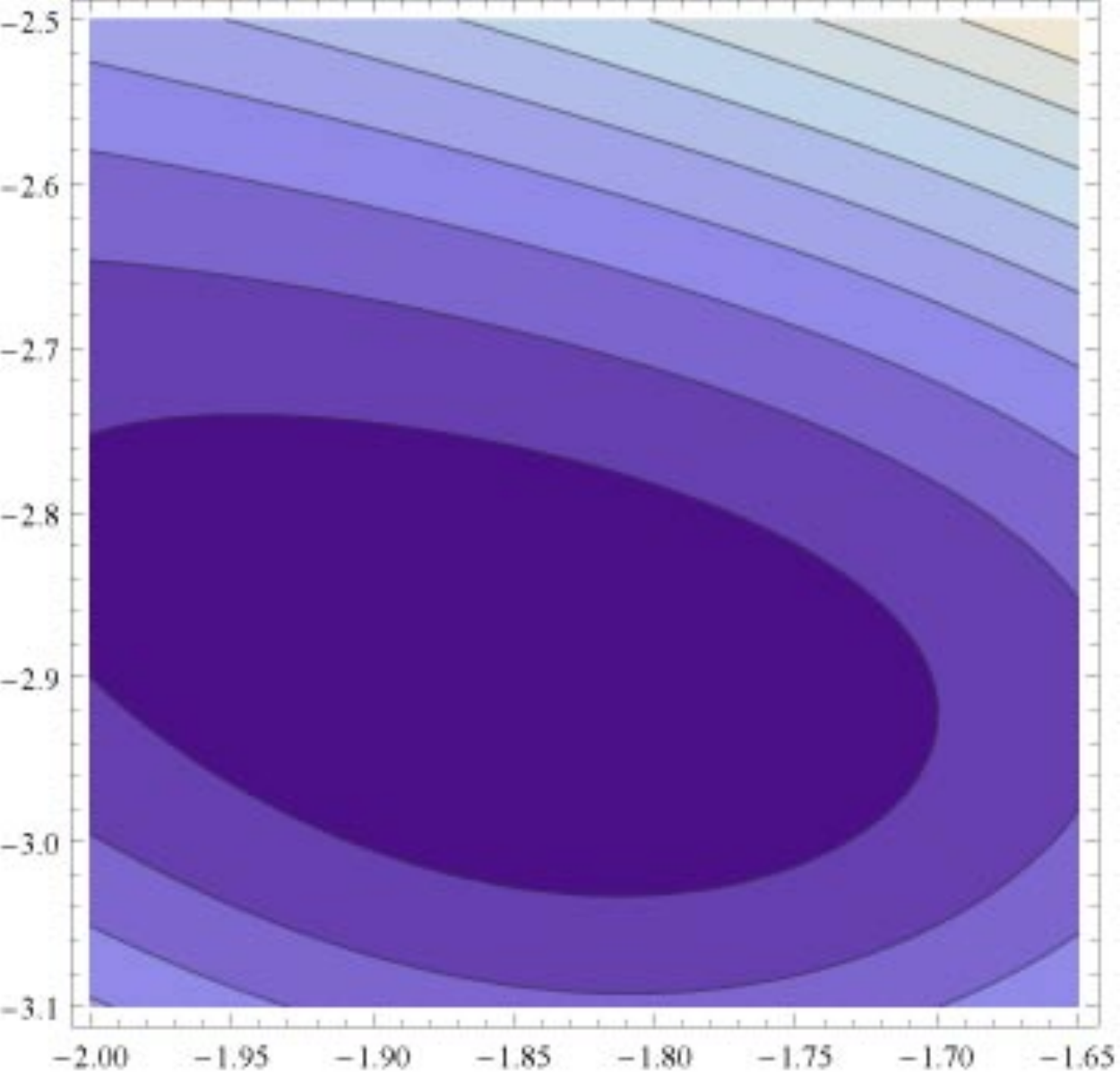}
\caption{}\label{fig:ex5pimin}
\end{subfigure}
\begin{subfigure}[b]{0.4\textwidth}
\centering
\includegraphics[width=0.9\textwidth]{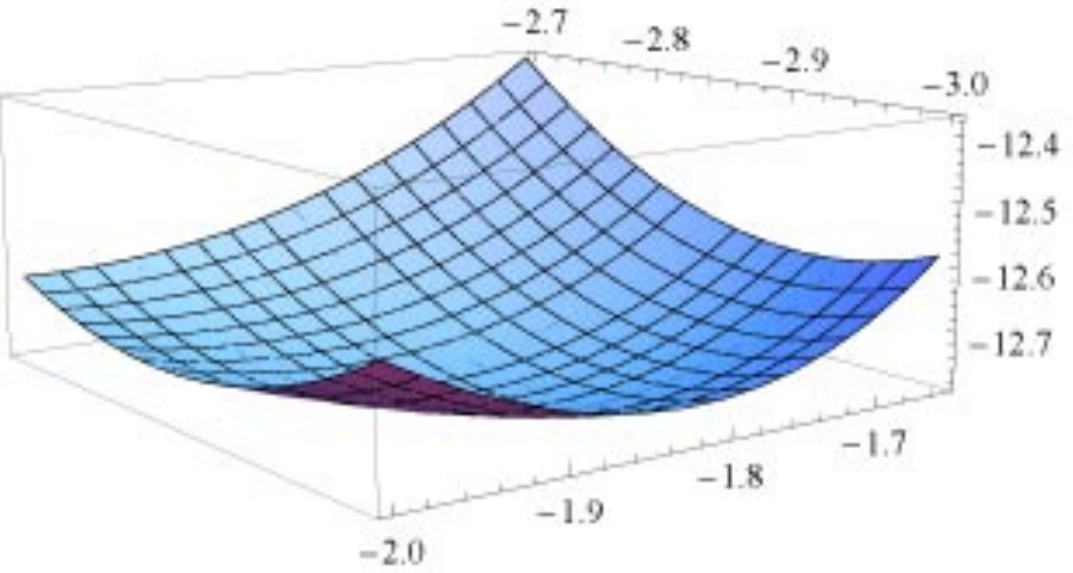}
\caption{}\label{fig:3dexample5pilocalmin}
\end{subfigure}
\caption{The double-min duality in Example 4: (a) contour plot of function $\vP^d(\tau,\sigma)$ near $(\bar\tau_3,\bar\sigma_3)$; (b) contour plot of function $\vP(x,y)$ near $(\bar x_3,\bar y_3)$; (c) graph of function $\vP(x,y)$ near $(\bar x_3,\bar y_3)$.}
\label{fig:ex1.5.min}
\end{figure}

\begin{figure}[H]
\centering
\begin{subfigure}[b]{0.3\textwidth}
\centering
\includegraphics[width=1.1\textwidth]{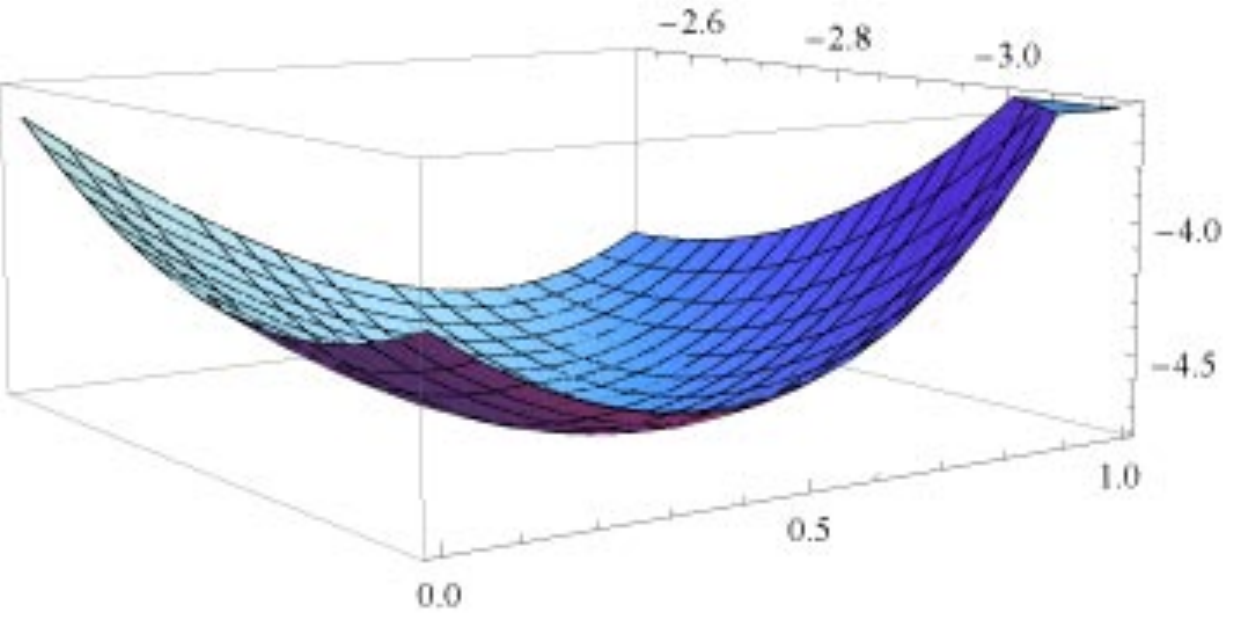}
\caption{}\label{fig:3Dex2localmin}
\end{subfigure}
\hspace{0.5in}
\begin{subfigure}[b]{0.3\textwidth}
\centering
\includegraphics[width=1.1\textwidth]{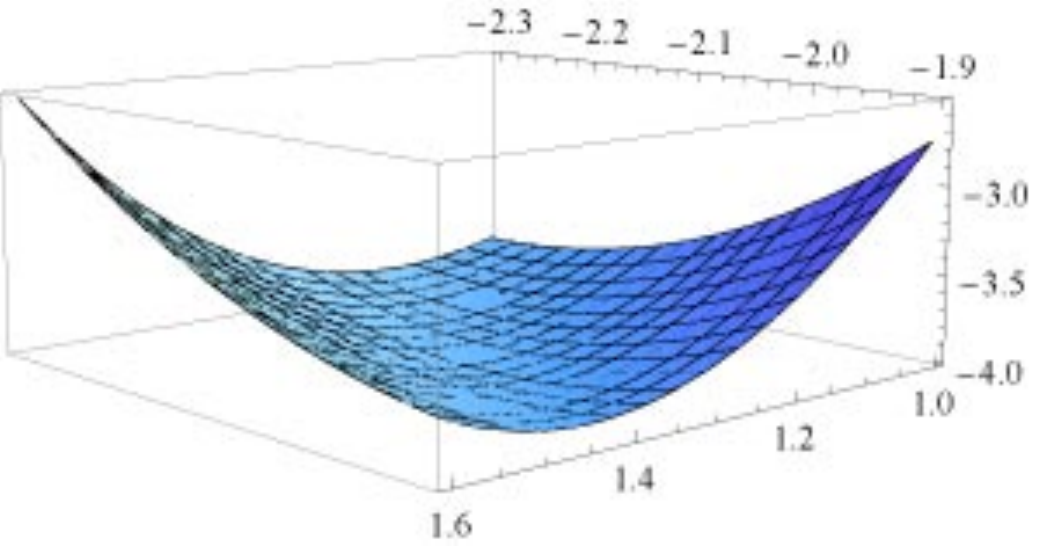}
\caption{}\label{fig:3Dex3localmin}
\end{subfigure}
\caption{graph of the primal problem near a local minimum point: (a) in Example 2; (b)in Example 3.}
\label{fig:localmin}
\end{figure}

\begin{thebibliography}{222}
\small{
\bibitem{ahn}
Ahn, J, Kuttler, KL, and Shillor, M. (2012).
Dynamic contact of two Gao beams, {\em Electron J Differ Equ},  194: 1-42.

\bibitem{cai-gao-qin}
Cai, K., Gao, DY, Qin, QH (2014). Post-buckling solutions of hyper-elastic beam by canonical dual finite element method,
{\em Mathematics and Mechanics of Solids},  19(6): 659-671

\bibitem{gao-chen2014}Chen, Y. and Gao, D.Y.(2016). Global solutions to nonconvex optimization of 4th-order polynomial and log-sum-exp functions. {\em  Journal of Global Optimization}, 64(3), 417-431.


\bibitem {ciarlet}Ciarlet, PG (2013). Linear and Nonlinear Functional Analysis with Applications, SIAM, Philadelphia.

\bibitem {eke-tem} Ekeland, I. and Temam, R. (1976). Convex Analysis and Variational Problems,
 North-Holland.


 \bibitem{fang-gao2007}Fang,S.-C., Gao,D.Y., Sheu,R.-L., Wu,S.Y.(2007). Canonical dual approach for solving 0-1 quadratic programming problems. {\em J. Industrial Management and Optimization}, 4(1): 125-142.

\bibitem{fang-gao2009}Fang,S.-C., Gao,D.Y., Sheu,R.-L., Xing,W.X.(2009). Global optimization for a class of fractional programming problems.{\em  Journal of Global Optimization}, 45(3): 337-353.


\bibitem{gao-mrc96} Gao, D.Y. (1996).
Nonlinear elastic beam theory with applications in contact problem and variational approaches,
{\em Mech. Research Commun.}, 23 (1): 11-17.

\bibitem{Gao00duality}Gao, D.(2000). Duality principles in nonconvex systems: theory, methods, and applications. Springer, New Yourk, 454p.


 \bibitem {gao-jogo00} Gao, D.Y.  (2000). Canonical dual transformation method and
generalized triality theory in nonsmooth global optimization, {\em
J. Global Optimization}, { 17} (1/4): 127-160.


\bibitem{gao-amma03}Gao, D.Y.(2003).
Nonconvex semi-linear problems and canonical dual solutions. {\em
Advances in Mechanics and Mathematics,} Vol. II, D.Y. Gao and R.W.
Ogden (ed), Springer,  pp. 261-312.


\bibitem{gao-opt03}Gao, D.Y.(2003).
Perfect duality theory and complete solutions to a class of global
optimization problems. {\em Optimization}, 52(4-5), 467-493.

\bibitem{gao2005}Gao,D.Y.(2005). Sufficient conditions and perfect duality in nonconvex minimization with inequality constraints. {\em Journal of Industrial and Management Optimization}, 1(1): 59-69.

\bibitem{gao2006}Gao,D.Y.(2006). Complete solutions and extremality criteria to polynomial optimization problems. {\em  Journal of Global Optimization}, 35: 131-143.

\bibitem{gao2007}Gao,D.Y.(2007). Solutions and optimality to box constrained nonconvex minimization problems. {\em Journal of Industrial and Management Optimization}, 3(2): 293-304.


 \bibitem{gao-cace09} Gao, D.Y. (2009).
 Canonical duality theory: unified understanding and generalized solutions for
global optimization. {\em Comput. \& Chem. Eng.} 33,
1964-1972.

\bibitem{gao-opl16} Gao, D.Y. (2016).
On unified modeling, canonical duality-triality theory, challenges and breakthrough in optimization,
 \url{http://arxiv.org/abs/1605.05534}.

\bibitem{gao-ruan2009}Gao,D.Y., Ruan,N.(2010). Solutions to quadratic minimization problems with box and integer constraints. {\em Journal of Global Optimization}, 47(3): 463-484.


\bibitem{bridgeMMS}Gao, D.Y., Ruan, N., Latorre, V. (2016). Canonical duality-triality theory: bridge between nonconvex analysis/mechanics and global optimization in complex systems.  {\em  Canonical Duality -Triality: Unified Theory for Multidisciplinary Study}, Springer.

\bibitem{gao-ruan2010} Gao, D.Y., Ruan, N., Sherali, H.D.(2010). Canonical duality solutions for fixed cost quadratic program. {\em Optimization and Optimal Control},  A. Chinchuluun et al. (eds.), Springer Optimization and Its Applications 39, 139-156.



\bibitem{gao-strang1989}Gao,D.Y., Strang,G.(1989). Geometric nonlinearity: Potential energy, complementary energy, and the gap function.{\em Quarterly Journal of Applied Mathematics}, XLVII(3): 487-504.

\bibitem{Gao-Wu-triality12}Gao, D., Wu, C.(2012). On the triality theory for a quartic polynomial optimization problem. {\em J. Ind. Manag.Optim.}, 8: 229-242.


\bibitem{gao-yu} Gao, D.Y., Yu, H.F. (2008).
Multi-scale modelling and canonical dual finite element method in
phase transitions of solids. \emph{ Int. J. Solids Struct.}, 45: 3660-3673.



\bibitem{hiri} Hiriart-Urruty, J.B.(1985).  Generalized differentiability, duality and optimization for problems dealing with
differences of convex functions. {\em Lecture Note Econ. Math. Syst.},  256: 37-70.



\bibitem{hor-tho} Horst, R., Thoai, N.V.(1999). DC Programming: overview.  {\em J. Opt. Theory Appl.}, 103: 1-43.


\bibitem{kuttler-etal} Kuttler, KL, Purcell, J, and Shillor, M. (2012).
 Analysis and simulations of a contact problem for a nonlinear dynamic beam with a crack.
{\em Q J Mech Appl Math},  65: 1-25.
\bibitem{marsd-hugh} Marsden,  J.E. and  Hughes, T.J.R.(1983).
 Mathematical Foundations of Elasticity,  Prentice-Hall.

\bibitem{ks} Kvasov D.E., Sergeyev Ya.D. (2013)  Lipschitz
     global optimization methods in control problems, {\em Automation
     and Remote Control,} 74(9), 1435-1448


\bibitem{l-l} Landau, L.D. and Lifshitz, E.M. (1976).
{\em Mechanics.} Vol. 1 (3rd ed.). Butterworth-Heinemann. ISBN 978-0-750-62896-9.

\bibitem{liu}
Liu, I.-S. (2005).
Further remarks on Euclidean objectivity and the principle of material frame-indifference.
{\em  Continuum Mech. Thermodyn}, 17: 125-133.

\bibitem{murd} Murdoch, A.I.(2003).
Objectivity in classical continuum physics: a rationale for discarding
the principle of invariance under superposed rigid body motions
in favour of purely objective considerations, {\em Continuum Mech. Thermodyn.}, 15: 309-320.

\bibitem{murd05}Murdoch, A.I.(2005).
On criticism of the nature of objectivity in classical continuum physics,
{\em Continuum Mech. Thermodyn.}, 17(2): 135-148.

\bibitem{pskz} Paulavicius R., Sergeyev Ya.D., Kvasov D.E.,   Zilinskas J. (2014) Globally-biased DISIMPL algorithm
     for expensive global optimization, {\em Journal
     of Global Optimization,} 59(2-3), 545-567.


\bibitem{tao-le} Pham Dinh Tao, Le Thi Hoai An(2014). Recent Advances in DC Programming and DCA.
{\em Transactions on   Computational Collective Intelligence},  13: 1-37.

\bibitem{tru-noll} Truesdell, C. and Noll, W. (1965).
{\em The Nonlinear Field Theories of Mechanics}, Springer-Verlag, 591pp.


\bibitem{toland} Toland, J.F.(1979). A duality principle for non-convex optimisation and the calculus of variations.
{\em  Arch.
Ration. Mech. Anal.},  71: 41-61.

\bibitem{tuy95}  Tuy, H.(1995). D.C. optimization: Theory, methods and algorithms. In: Horst, R., Pardalos, P.M. (eds.)
{\em Handbook of Global Optimization},  pp. 149-216. Kluwer Academic Publishers, Dordrecht.




\bibitem{santos-gao}  Santos, H.A.F.A. and Gao D.Y. (2011) Canonical dual finite element method for solving post-buckling problems of a large deformation elastic beam, {\em Int. J. Nonlinear Mechanics, } 47: 240 - 247.

\bibitem{s-s} Strongin R.G., Sergeyev Ya.D. (2000)
{\em Global optimization with non-convex constraints: Sequential and parallel algorithms,} Kluwer Academic Publishers, Dordrecht. Springer (3rd ed. 2014),  728 pp.

\bibitem{wang-fang-gao2008}Wang Z.B., Fang, S.-C., Gao, D.Y., Xing W.X.(2008). Global extremal conditions for multi-integer quadratic programming. {\em Journal of Industrial and Management Optimization}, 4(2): 213-225.












 }

\end{thebibliography}
\end{document}